\documentclass[12pt]{amsart}
\usepackage[latin1]{inputenc}
\usepackage[english]{babel}
\usepackage{amsmath,amssymb,amsthm,amsfonts, color}
\usepackage{enumerate}
\usepackage{graphicx}
\usepackage{latexsym}
\usepackage[all]{xy}
\usepackage{tikz}

\makeatletter
\@namedef{subjclassname@2010}{ \textup{2010} Mathematics Subject Classification}
\makeatother

\newtheorem{theorem}{Theorem}
\newtheorem{corollary}{Corollary}
\newtheorem{proposition}{Proposition}
\newtheorem{remark}{Remark}

\theoremstyle{definition}

\numberwithin{equation}{section}

\numberwithin{equation}{section}
\DeclareMathOperator{\supp}{supp}

\DeclareMathOperator*{\esssup}{ess\,sup}

\begin{document}
\title{Isomorphic structure of Ces\`aro and Tandori spaces}
\thanks{{\rm *}The first author has been partially supported by the Ministry of Education and Science 
of the Russian Federation and the second author has been partially supported by the grant 04/43/DSPB/0086
from the Polish Ministry of Science and Higher Education.}
\author[Astashkin]{Sergey V. Astashkin}
\address[Sergey V. Astashkin]{Department of Mathematics and Mechanics\\
Samara State University, Acad. Pavlova 1, 443011 Samara, Russia\\ and
Samara State Aerospace University (SSAU), Moskovskoye shosse 34, 443086, Samara, Russia}
\email{\texttt{astash@samsu.ru}}
\author[Le\'snik]{Karol Le\'snik}
\address[Karol Le{\'s}nik]{Institute of Mathematics\\
of Electric Faculty Pozna\'n University of Technology, ul. Piotrowo 3a, 60-965 Pozna{\'n}, Poland}
\email{\texttt{klesnik@vp.pl}}
\author[Maligranda]{Lech Maligranda}
\address[Lech Maligranda]{Department of Engineering Sciences and Mathematics\\
Lule{\aa} University of Technology, SE-971 87 Lule{\aa}, Sweden}
\email{\texttt{lech.maligranda@ltu.se}}
\maketitle

\vspace{-7mm}

\begin{abstract}
We investigate the isomorphic structure of the Ces\`aro spaces and their duals, the Tandori spaces. The main result 
states that the Ces\`aro function space $Ces_{\infty}$ and its sequence counterpart $ces_{\infty}$ are isomorphic, 
which answers to the question posted in \cite{AM09}. This is rather surprising since  $Ces_{\infty}$ has no natural 
lattice predual similarly as the known Talagrand's example \cite{Ta81}. We prove that neither $ces_{\infty}$ is isomorphic 
to $l_{\infty}$ nor $Ces_{\infty}$ is isomorphic to the Tandori space $\widetilde{L_1}$ with the norm 
$\|f\|_{\widetilde{L_1}}= \|\widetilde{f}\|_{L_1},$ where $\widetilde{f}(t):=  \esssup_{s \geq t} |f(s)|.$  Our investigation involves also
an examination of the Schur and Dunford-Pettis properties of Ces\`aro and Tandori spaces. In particular, using Bourgain's 
results we show that a wide class of Ces{\`a}ro-Marcinkiewicz and Ces{\`a}ro-Lorentz spaces have the latter property. 
\end{abstract}


\section{\protect \medskip Introduction and contents}

Most commonly the classical Ces\`aro spaces appear as optimal domains of the Ces\`aro (Hardy) operator or some its versions 
(see \cite{DS07}, \cite{NP10}, \cite{LM15b}). Moreover, they can coincide with the so-called down spaces introduced and 
investigated by Sinnamon (see \cite{KMS07},  \cite{MS06}, \cite{Si94}, \cite{Si01}, \cite{Si07}), but having their roots in the 
papers of Halperin and Lorentz. 
Comparing to the function case, there is much more rich literature devoted to Ces\`aro sequence spaces and their duals
(see the classical paper of Bennett \cite{Be96} and also \cite{CH01}, \cite{CMP00}, \cite{Ja74}, \cite{KK12}, \cite{MPS07}).
Development of this topic related to the weighted case including the so-called blocking technique one can find in the book \cite{GE98}. 

In this paper we investigate the isomorphic structure of abstract Ces\`aro spaces $CX$ and their duals, Tandori spaces 
$\widetilde{X}$ on three separable measure spaces ${\mathbb N}, [0, \infty)$ and $[0, 1]$. For a Banach ideal space 
$X$ of measurable functions on  $I=[0, \infty)$ or $I=[0, 1]$, $CX$ is defined as the space of all measurable functions $f$ 
on $I$ such that $C|f| \in X$, equipped with the norm $\| f \|_{CX} : = \| C|f| \|_X$, where $C$ denotes the Ces\`aro operator, i.e., 
$(Cf)(x): = \frac{1}{x} \int_0^x f(t) \, dt$ for $x\in I$. In the case of a Banach ideal sequence space $X$, in the definition of 
the Ces\`aro space it is used the corresponding discrete Ces\`aro operator $(C_d x)_n: = \frac{1}{n} \sum_{k=1}^n x_k,$ 
$n \in \mathbb N$.

Study of abstract Ces\`aro function spaces, under this name, started in the paper \cite{LM15a}, where a description of their 
K\"othe duals by the so-called Tandori spaces was found. It is worth to note here that the obtained results substantially 
differ in the cases $I = [0, \infty)$ and $I = [0, 1]$. Continuing the same direction of research, in \cite{LM16}, the authors 
have examined interpolation properties of these spaces. 

Investigation of the isomorphic structure of classical Ces\`aro function spaces $Ces_p: = CL_p$ was initiated in the paper 
\cite{AM09} (see also \cite{AM14}); at the same time, studying classical Ces\`aro sequence spaces $ces_p: = Cl_p$ started 
much earlier (see \cite{MPS07} and the references cited therein). Among other things, in \cite{AM09}, the existence of an 
isomorphism between the spaces $Ces_p[0, \infty)$ and $Ces_p[0, 1]$ for all $1 < p \leq \infty$ has been proved. On the 
other hand, $Ces_p(I)$ and $ces_p$ for any $1 < p < \infty$ are clearly not isomorphic because, in contrast to $ces_p$, 
the space $Ces_p(I)$ is not reflexive.

Therefore, the only remained question (already formulated in \cite{AM09} and \cite{AM14}) was whether $Ces_{\infty}$ is
isomorphic to $ces_{\infty}$ or not. Theorem \ref{mainthm}, one of the main results of the present paper, solves this problem 
in affirmative. It is instructive to compare this result with  the well-known Pe{\l}czy\'nski theorem on the existence of 
an isomorphism between the spaces $L_{\infty}$ and $l_{\infty}$ \cite{Pe58} and also with Leung's  result which showed, 
using Pe{\l}czy\'nski decomposition method, that the weak $L_p$--spaces: $L_{p, \infty}[0, 1]$, $L_{p, \infty}[0, \infty)$ and 
$l_{p, \infty}$ for every $1 < p < \infty$ are all isomorphic  \cite{Le93}. 

On the other hand, we prove that  $ces_{\infty}$ and  $Ces_{\infty}(I)$ are not isomorphic to $l_{\infty}$, and to the Tandori 
space $\widetilde{L_1}(I)$ with the norm $\|f\|_{\widetilde{L_1}}= \|\widetilde{f}\|_{L_1}$, where 
$\widetilde{f}(t):= \esssup_{s \in I, \, s \geq t} |f(s)|$, 
respectively. Moreover, if $X$ is a reflexive symmetric space on $[0, 1]$ and the Ces\`aro operator $C$ is bounded on $X$, 
then $CX$ is not isomorphic to any symmetric space on $[0, 1]$. The main tool in proving these results is coming from the fact 
that either of the Ces\`aro spaces $ces_{\infty}$ and $Ces_{\infty}(I)$ contains a complemented copy of $L_1[0, 1]$ (see 
Proposition \ref{Pro1n}) but the other ones do not. We make use also of a characterization due to Hagler-Stegall \cite{HS73} 
of dual Banach spaces containing complemented copies of $L_1[0,1]$ and Cembranos-Mendoza result \cite{CM08}, stated 
that the mixed-norm space $l_{\infty}(l_1)$ contains a complemented copy of $L_1[0, 1]$ while the space $l_1(l_{\infty})$ 
does not.
 
Along with the isomorphic structure of abstract Ces\`aro and Tandori spaces  we study in this paper also their Schur and 
Dunford-Pettis properties being isomorphic invariants. In particular, we are able to find a new rather natural Banach space 
non-isomorhic to $l_1$ with the Schur property, namely, the sequence Tandori space $\widetilde{l_1}$ with the norm 
$\|(a_k)\|_{\widetilde{l_1}}= \|(\widetilde{a_k})\|_{l_1},$ where $\widetilde{a_k}:=\sup_{i\ge k}|a_i|.$ Regarding the 
Dunford-Pettis property we note that, generally, it is not easy to find out whether a given space has this property. 
We apply here a deep Bourgain's result \cite{Bo81}, which shows that every $l_{\infty}$-sum of $L_1$-spaces has the 
Dunford-Pettis property. Basing on this, Bourgain deduced also that the spaces of vector-valued functions $L_1(\mu,C(K))$ 
and $C(K, L_1(\mu))$, where $\mu$ and $K$ are a $\sigma$-finite measure and any compact Hausdorff set, respectively, 
and their duals have the Dunford-Pettis property. Using these facts and a suitable description of Ces{\`a}ro and Tandori 
spaces, obtained in this paper, we prove that $Ces_{\infty}(I), \widetilde{L_1}(I)$ and, under some conditions on dilation 
indices of a function $\varphi,$ Ces{\`a}ro-Marcinkiewicz spaces $CM_{\varphi}[0, \infty)$,  their separable parts 
$C(M_{\varphi}^0)(I)$, Ces{\`a}ro-Lorentz spaces $C\Lambda_{\varphi}(I)$ as well as Tandori-Lorentz spaces 
$\widetilde{\Lambda_{\varphi}}(I)$ in both cases $I=[0,\infty)$ and $I=[0,1]$ all enjoy the Dunford-Pettis property. Recall 
that, in  \cite{KM00}, Kami\'nska and Masty{\l}o proved that $l_1, c_0$ and $l_{\infty}$ are the only symmetric sequence 
spaces  with the Dunford-Pettis property and there exist exactly six non-isomorphic symmetric spaces on $[0, \infty)$ enjoying 
the latter property: $L_1, L_{\infty}, L_1 \cap L_{\infty}, L_1 + L_{\infty}$ and the closures of $L_1 \cap L_{\infty}$ in $L_{\infty}$ 
and in  $L_1 + L_{\infty}$. 

The paper is organized as follows. Following this Introduction, Section 2 collects some necessary preliminaries, firstly, on Banach 
ideal and symmetric spaces and, secondly, on Ces{\`a}ro and Tandori spaces. Here, we recall also Theorem A related to the duality 
from \cite{LM15a} and prove Proposition 1 on the existence of a complemented copy of $L_1[0, 1]$ in an arbitrary Ces{\`a}ro space 
$CX$ provided the Ces{\`a}ro operator $C$ is bounded in $X$. These results are frequently used throughout the paper. 

Section 3 contains results related to studying the Schur and Dunford-Pettis properties of Tandori and Ces{\`a}ro sequence spaces. 
We proved that $\widetilde{l_1}$ has the Schur property (Theorem 1) and $ces_{\infty}$ contains a complemented copy of 
$L_1[0, 1]$ (Theorem 3). Moreover, we investigate the conditions under which Ces{\`a}ro-Marcinkiewicz and Ces{\`a}ro-Lorentz 
sequence spaces and also their duals have the Dunford-Pettis property (see Theorems \ref{Thm4a} and \ref{Thm5a}). Finally, we 
show that the spaces $CX$ and $\widetilde{X}$ fail to have the Dunford-Pettis property whenever a symmetric sequence space 
$X$ is reflexive and the discrete Ces{\`a}ro operator is bounded in $X$ or in $X'$, respectively (Theorem \ref{Thm4}).

Section 4 deals with the Dunford-Pettis property of Ces{\`a}ro and Tandori function spaces. It is proved that, under the assumption 
$q_{\varphi} < 1$, both Tandori-Lorentz space $\widetilde{\Lambda_{\varphi}}[0, \infty)$ and Ces{\`a}ro-Marcinkiewicz space 
$CM_{\varphi}[0, \infty)$ have the Dunford-Pettis property (Theorem \ref{Thm5}). In particular, two non-isomorphic spaces 
$Ces_{\infty}(I)$ and $\widetilde{L_1}(I)$ have the latter property (see Theorem \ref{Thm6}). Similar result holds also for the separable 
parts of the Ces{\`a}ro-Marcinkiewicz spaces $CM_{\varphi}[0, \infty)$ and $CM_{\varphi}[0, 1]$ provided 
$\lim_{t \rightarrow 0^+} \varphi(t) = 0$ and $q_{\varphi} < 1$ or $q_{\varphi}^0 < 1$, respectively (Theorem \ref{Thm7n} and 
Theorem \ref{Thm9new}). Moreover, if $X$ is a reflexive symmetric function space satisfying some conditions, then $CX$ and 
$\widetilde{X}$ fail to have the Dunford-Pettis property (Theorem \ref{Thm7new}).

The last Section 5 contains one of the main results of the paper, showing that the spaces $Ces_{\infty}$ and $ces_{\infty}$ are 
isomorphic (Theorem \ref{mainthm}). This gives a positive answer to the question posed in \cite[Problem 1]{AM09} and repeated 
in \cite[Problem 4]{AM14}. An interesting consequence of this result is the fact that the space $Ces_{\infty}$ is isomorphic to a dual 
space although $[(Ces_{\infty})^{\prime}]^0 = (\widetilde{L_1})^0 = \{0\}$ (Corollary \ref{Cor7n}) and so there is no natural candidate 
for its predual (for $ces_{\infty}$, however, the predual is $\widetilde{l_1}$ because $(\widetilde{l_1})^* = 
(\widetilde{l_1})^{\prime} = ces_{\infty})$.
We explain that this phenomenon has its counterpart in the general theory of Banach lattices, discussing its relation to 
Lotz's result \cite{Lo75} and to Talagrand's example of a separable  Banach lattice being a dual space (and hence having 
the Radon-Nikodym property) such that for each $x^* \in E^*$, the interval $[0, |x^*|]$ is not weakly compact \cite{Ta81}
(see Proposition \ref{Cor12}). Finally, we prove that $Ces_{\infty}(I)$ is isomorphic to the space 
$(\bigoplus_{k=1}^{\infty} {\mathcal M}[0,1])_{l_\infty}$, where ${\mathcal M}[0,1]$ is the space of regular Borel measures on $[0,1]$ 
of finite variation (Theorem \ref{Thm12}).

\section{\protect Definitions and basic facts}

\subsection{\protect Banach ideal spaces and symmetric spaces}
By $L^0 = L^0(I)$ we denote the set of all equivalence 
classes of real-valued Lebesgue measurable functions defined on $I = [0, 1]$ or $I = [0, \infty)$. A {\it Banach ideal space} 
$X = (X, \|\cdot\|)$ on $I$ is understood as a Banach space contained in $L^0(I)$, which satisfies the so-called ideal 
property: if $f, g \in L^0(I), |f| \leq |g|$ almost everywhere (a.e.) with respect to the Lebesgue measure
on $I$ and $g \in X$, then $ f\in X$ and $\|f\| \leq \|g\|$. Sometimes we write 
$\|\cdot\|_{X}$ to be sure which norm in the space is taken. If it is not stated otherwise we suppose that a Banach ideal 
space $X$ contains a function $f\in X$ with $f(x) > 0$ a.e. on $I$ (such a function is called the {\it weak unit} in $X$), 
which implies that ${\rm supp}X = I$. Similarly we define a Banach ideal sequence space (i.e., on $I = \mathbb N$ with the counting
measure).

Since an inclusion of two Banach ideal spaces is continuous, we prefer to write in this case $X\hookrightarrow Y$ rather that $X\subset Y$.
Moreover, the symbol $X\overset{A}{\hookrightarrow }Y$ indicates that $X\hookrightarrow Y$ with the norm of the inclusion operator 
not bigger than $A$, i.e., $\| f\|_{Y} \leq A \|f\|_{X}$ for all $f\in X$. Also, $X = Y$ (resp. $X \equiv Y$) means that the spaces $X$ and $Y$ 
have the same elements with equivalent (resp. equal) norms. By $X \simeq Y$ we denote the fact that the Banach spaces $X$ and $Y$ 
are isomorphic.

For a Banach ideal space $X = (X, \|\cdot\|)$ on $I$ the {\it K{\"o}the dual space} (or {\it associated space}) $X^{\prime}$ is the space 
of all $f \in L^0(I)$ such that the {\it associated norm}
\begin{equation} \label{dual}
\|f\|^{\prime} = \sup_{g \in X, \, \|g\|_{X} \leq 1} \int_{I} |f(x) g(x) | \, dx
\end{equation}
is finite. The K{\"o}the dual $X^{\prime} = (X^{\prime}, \|\cdot \|^{\prime})$ is a Banach ideal space. Moreover, 
$X \overset{1}{\hookrightarrow }X^{\prime \prime}$ and we have equality $X = X^{\prime \prime}$ with $\|f\| = \|f\|^{\prime \prime}$ 
if and only if the norm in $X$ has the {\it Fatou property}, that is, if the conditions $0 \leq f_{n} \nearrow f$ a.e. on $I$
and $\sup_{n \in {\bf N}} \|f_{n}\| < \infty$ imply that $f \in X$ and $\|f_{n}\| \nearrow \|f\|$.

For a Banach ideal space $X = (X, \|\cdot\|)$ on $I$ with the K\"othe dual $X^{\prime}$ the following {\it generalized H\"older-Rogers 
inequality} holds: if $f \in X$ and $g \in X^{\prime}$, then $f g$ is integrable and
\begin{equation} \label{HRI}
\int_I |f(x) g(x)|\, dx \leq \| f\|_X \, \| g\|_{X^{\prime}}
\end{equation}

A function $f$ in a Banach ideal space $X$ on $I$ is said to have an {\it order continuous norm} in $X$ if for any
decreasing sequence of Lebesgue measurable sets $A_{n} \subset I $ with $m(\bigcap_{n=1}^{\infty} A_n) = 0$, where 
$m$ is the Lebesgue measure, we have $\|f \chi_{A_{n}} \| \rightarrow 0$ as $n \rightarrow \infty$. The set of all functions 
in $X$ with order continuous norm is denoted by $X^0$. If $X^0 = X$, then the space $X$ is said to be {\it order continuous}.
For an order continuous Banach ideal space $X$ the K{\"o}the dual $X^{\prime}$ and the dual space $X^{*}$ coincide.
Moreover, a Banach ideal space $X$ with the Fatou property is reflexive if and only if both $X$ and its K\"othe dual 
$X^{\prime}$ are order continuous.

For a weight $w(x)$, i.e., a measurable function on $I$ with $0 < w(x) < \infty$ a.e. and for a Banach ideal space $X$ 
on $I$, the {\it weighted Banach ideal space} $X(w)$ is defined as the set $X(w)=\{f\in L^0: fw \in X\}$ with the norm 
$\|f\|_{X(w)}=\| f w \|_{X}$. Of course, $X(w)$ is also a Banach ideal space and 
$[X(w)]^{\prime} \equiv X^{\prime}(1/w)$.

A Banach ideal space $X = (X,\| \cdot \|_{X})$ is said to be a {\it symmetric} (or {\it rearrangement invariant}) space on $I$ 
if from the conditions $f\in X$, $g \in L^{0}(I)$ and the equality  $d_{f}(\lambda)=d_{g}(\lambda)$ for all $\lambda>0$, where
$$d_{f}(\lambda) := m(\{x \in I: |f(x)|>\lambda \}),\lambda \geq 0,$$ 
it follows that $g\in X$ and $\| f\|_{X} = \| g\|_{X}$. In particular, $\| f\|_{X}=\| f^{\ast }\|_{X}$, where 
$f^{\ast }(t) = \mathrm{\inf } \{\lambda >0\colon \ d_{f}(\lambda ) < t\},\ t\geq 0$. 

For a symmetric function space $X$ on $I$ its fundamental function $\varphi_X$ is defined as follows
\begin{equation*}
\varphi_X(t)=\|\chi_{[0,t]}\|_X, ~t>0,
\end{equation*}
where by $\chi_E$ throughout will be denoted the characteristic function of a set $E.$

Let us recall some classical examples of symmetric spaces. Each increasing concave function $\varphi$ on $I, \varphi (0) = 0,$ 
generates the {\it Lorentz space} $\Lambda_{\varphi}$ (resp. {\it Marcinkiewicz space} $M_{\varphi}$) on $I$ endowed with 
the norm
\begin{equation*}
\|f\|_{\Lambda_{\varphi}} = \int_I f^*(s) \,d\varphi(s).
\end{equation*}
(resp.
\begin{equation} \label{Marcinkiewicz}
\|f\|_{M_{\varphi}}=\sup_{t\in I} \frac{\varphi(t)}{t} \int_0^t f^*(s) \,ds).
\end{equation}
In the case $\varphi(t) = t^{1/p}, 1 < p < \infty$, the Marcinkiewicz space is also called the weak-$L_p$ space (shortly denoted by 
$L_{p, \infty}$) and the norm (\ref{Marcinkiewicz}) is equivalent to the quasi-norm $\| f \|_{L_{p, \infty}} = \sup_{t\in I} t^{1/p} f^*(t)$.
In general, a space $M_{\varphi}$ is not separable (for example, when $\lim_{t \rightarrow 0^+} \frac{t}{\varphi(t)} = 
\lim_{t \rightarrow \infty} \frac{\varphi(t)}{t} = 0$), but the spaces
\begin{equation*}
\{f \in M_{\varphi}: \lim_{t \rightarrow 0^+, \infty} \frac{\varphi(t)}{t} \int_0^tf^*(s) \,ds = 0 \} ~~ {\rm in ~ the ~ case} ~ I = [0, \infty)
\end{equation*}
and
\begin{equation*}
\{f \in M_{\varphi}: \lim_{t \rightarrow 0^+} \frac{\varphi(t)}{t} \int_0^tf^*(s) \,ds = 0\} ~~ {\rm in ~ the ~ case} ~ I = [0,1]
\end{equation*}
with the norm (\ref{Marcinkiewicz}) are the separable symmetric spaces which, in fact, coincide with the space 
$M_{\varphi}^0$ on $I=[0,\infty)$ or $I=[0,1]$, respectively, provided $\lim_{t \rightarrow 0^+} \varphi(t) = 0$ (cf. \cite[pp. 115-116]{KPS82}). 

Let $\Phi$ be an increasing convex function on $[0, \infty)$ such that $\Phi (0) = 0$. Denote by $L_{\Phi}$ the {\it Orlicz space}  
on $I$ (see e.g. \cite{KR61}, \cite{Ma89}) endowed with the Luxemburg-Nakano norm
$$
\| f \|_{L_{\Phi}} = \inf \{\lambda > 0: \int_I \Phi(|f(x)|/\lambda) \, dx \leq 1 \}.
$$

For a given symmetric space $X$ with the fundamental function $\varphi$ (every such a function is equivalent to a concave function) 
we have
\begin{equation*}
\Lambda_{\varphi} \overset{2}{\hookrightarrow } X \overset{1}{\hookrightarrow } M_{\varphi} ~ 
{\rm and} ~(M_{\varphi})^{\prime} = \Lambda_{\psi} ~ {\rm with}  ~ \psi(t) = \frac{t}{\varphi(t)}, t > 0.
\end{equation*}

Similarly one can define Banach ideal and symmetric sequence spaces and all the above notions. In particular, the fundamental 
function of a symmetric sequence space $X$ is the function $\varphi_X(n) = \| \sum_{k=1}^n e_k \|_X, n \in \mathbb N,$ where 
$\{e_n\}_{n=1}^\infty$ is the canonical basic sequence of $X$. Moreover, the {\it Lorentz sequence space} $\lambda_{\varphi}$ (resp. 
{\it Marcinkiewicz sequence space} $m_{\varphi}$) is defined as the space of all 
sequences $x = (x_n)_{n=1}^{\infty},$ for which the following norm is finite
\begin{equation*} 
\|x \|_{\lambda_{\varphi}}=\sum_{k=1}^\infty x_k^* (\varphi(k+1)-\varphi(k)) 
\end{equation*}
(resp.
\begin{equation} \label{marcinkiewicz}
\|x \|_{m_{\varphi}}=\sup_{n \in \mathbb N} \frac{\varphi(n)}{n}\sum_{k=1}^n x_k^*), 
\end{equation}
where $\varphi$ is an increasing concave function on $[0, \infty)$ and $(x_k^*)$ is the decreasing rearrangement of 
the sequence $(|x_k|)_{k=1}^{\infty}$. In the case $\varphi(n) = n^{1/p}, 1 < p < \infty$, the Marcinkiewicz space $m_{\varphi}$ is also 
called the weak--$l_p$ space (shortly denoted by $l_{p, \infty}$) and the norm (\ref{marcinkiewicz}) is equivalent to the quasi-norm 
$\| x \|_{l_{p, \infty}} = \sup_{k \in \mathbb N} k^{1/p} x_k^*$.

The {\it dilation operators} $\sigma_s$ ($s > 0$) defined on $L^0(I)$ by $\sigma_s f(x) = f(x/s)$ if $I = [0, \infty)$ and
$\sigma_s f(x) = f(x/s) \chi_{[0, \, \min(1, \, s)]}(x)$ if $I = [0, 1]$ are bounded in any symmetric space $X$ on $I$ and 
$\| \sigma_s \|_{X \rightarrow X} \leq \max (1, s)$ (see \cite[p. 148]{BS88} and \cite[pp. 96-98]{KPS82}). 
These operators are also bounded in some Banach ideal spaces which are not symmetric. For example, if $X = L_p(x^{\alpha})$, 
then $\| \sigma_s\|_{X \rightarrow X} = s^{1/p + \alpha}$ (see \cite{Ru80} for more examples). 
The {\it Boyd indices} of a symmetric space $X$ are defined by
\begin{equation*}
\alpha_X = \lim_{s \rightarrow 0^+} \frac{\ln \| \sigma_s\|_{X \rightarrow X} }{\ln s}, \beta_X = 
\lim_{s \rightarrow \infty} \frac{\ln \| \sigma_s\|_{X \rightarrow X} }{\ln s},  
\end{equation*}
and we have $0 \leq \alpha_X \leq \beta_X \leq 1$ (cf. \cite[pp. 96-98]{KPS82} and \cite[p. 139]{LT79}).

For every $m \in \mathbb N$ let $\sigma_m$ and $\sigma_{1/m}$ be the {\it dilation operators} defined in spaces of 
sequences $a = (a_n)$ by:
$$
\sigma_m a = \left( ( \sigma_m a)_n \right)_{n=1}^{\infty} = \big (a_{[\frac{m-1+n}{m}]} \big)_{n=1}^{\infty} 
= \big ( \overbrace {a_1, a_1, \ldots, a_1}^{m}, \overbrace {a_2, a_2, \ldots, a_2}^{m}, \ldots \big)
$$
\begin{eqnarray*}
\sigma_{1/m} a 
&=& 
\left( ( \sigma_{1/m} a)_n \right)_{n=1}^{\infty} = \Big (\frac{1}{m} \sum_{k=(n-1)m + 1}^{nm} a_k \Big)_{n=1}^{\infty} \\
&=& 
\big ( \frac{1}{m} \sum_{k=1}^m a_k, \frac{1}{m} \sum_{k=m+1}^{2m} a_k, \ldots, \frac{1}{m} \sum_{k=(n-1)m + 1}^{nm} a_k, \ldots \big)
\end{eqnarray*}
(cf. \cite[p. 131]{LT79} and \cite[p. 165]{KPS82}). 
They are bounded in any symmetric sequence space and also in some non-symmetric 
Banach ideal sequence spaces; for example, $\| \sigma_{m} \|_{l_p(n^{\alpha}) \rightarrow l_p(n^{\alpha})} \leq m^{1/p} \max(1, m^{\alpha})$ 
and $\| \sigma_{1/m} \|_{l_p(n^{\alpha}) \rightarrow l_p(n^{\alpha})} \leq m^{-1/p} \max(1, m^{-\alpha})$.

The {\it lower index $p_{\varphi}$ and upper index} $q_{\varphi}$ of an arbitrary positive function $\varphi$ on $[0, \infty)$ are defined as 
\begin{equation} \label{indices}
p_{\varphi} = \lim_{s \rightarrow 0^{+}} \frac{\ln \overline \varphi( s) }{\ln s}, ~ q_{\varphi} = \lim_{s \rightarrow \infty } \frac{\ln \overline \varphi(s) }{\ln s}, 
~ {\rm where} ~~ {\overline \varphi}( s) = \sup_{t > 0 } \frac{\varphi( st ) }{\varphi( t) }. 
\end{equation}
It is known (see, for example, \cite{KPS82} and \cite{Ma85, Ma89}) that for a concave function $\varphi $ on $[0, \infty)$ 
we have $0 \leq p_{\varphi} \leq q_{\varphi} \leq 1$. Moreover, the estimate 
\begin{equation} \label{estimate2.13}
\int_0^t \frac{1}{\varphi(s)} \, ds \leq C \frac{t}{\varphi(t)} ~ {\rm for ~all} ~~ t >0
\end{equation}
is equivalent to the condition $q_{\varphi} < 1$ (cf. \cite[Lemma 1.4]{KPS82}, \cite[Theorem 11.8]{Ma85} and \cite[Theorem 6.4(a)]{Ma89}). 

If an increasing concave function $\varphi$ is defined on $[0, 1]$ (resp. on $[1,\infty)$), 
then the corresponding indices $p_{\varphi}^0, q_{\varphi}^0$ (resp. $p_{\varphi}^\infty, q_{\varphi}^\infty$) are 
the numbers defined as limits in (\ref{indices}), where instead of ${\overline \varphi}$ we take the function 
${\overline \varphi^{0}}(s) = \sup_{0 < t \leq \min(1, 1/s)} \frac{\varphi( st ) }{\varphi( t)}$ 
(resp. ${\overline \varphi^{\infty}}(s) = \sup_{t \geq \max(1, 1/s)} \frac{\varphi( st ) }{\varphi( t)}$). Of course, 
$0 \leq p_{\varphi}^0 \leq q_{\varphi}^0 \leq 1$ (resp. $0 \leq p_{\varphi}^\infty \leq q_{\varphi}^\infty \leq 1$)  
and estimate (\ref{estimate2.13}) for all $0 < t \leq 1$ is equivalent to the condition $q_{\varphi}^0 < 1$. 

If $X_n,$ $n=1,2,\dots,$ are Banach spaces and $1 \leq p\leq \infty$, then
$(\bigoplus_{n=1}^{\infty} X_n)_{l_p}$ is the Banach space of all sequences $\{x_{n}\},$ $x_n\in X_n,$ $n=1,2,\dots,$ 
such that
$$
\| \{x_{n}\} \| := \Big(\sum_{n = 1}^{\infty} \|x_n\|^p\Big)^{1/p} < \infty,
$$
with natural modification in the case when $p$ is infinite.

For general properties of Banach ideal and symmetric spaces we refer to the books \cite{BS88}, \cite{KA77}, \cite{KPS82}, 
\cite{LT79} and \cite{Ma89}. 

\subsection{\protect \medskip Ces\`aro and Tandori spaces}

The Ces\`aro and Copson operators $C$ and $C^*$ are defined, respectively, as
$$
Cf(x) = \frac{1}{x} \int_0^x f(t) \,dt, 0 < x \in I\;\;\mbox{and}\;\; C^*f(x) = \int_{I \cap [x, \infty)} \frac{f(t)}{t} \,dt,  x \in I.
$$
By $\widetilde{f}$ we will understand the decreasing majorant of a given function $ f$, i.e.,
$$
\widetilde{f}(x) = \esssup_{t \in I, \, t \geq x} |f(t)|.
$$
For a Banach ideal space $X$ on $I$ we define the {\it abstract Ces\`aro space} $CX = CX(I)$ as 
\begin{equation} \label{Cesaro}
CX = \{f\in L^0(I): C|f| \in X\} ~~ {\rm with ~the ~norm } ~~ \|f\|_{CX} = \| C|f| \|_{X},
\end{equation}
and the {\it abstract Tandori space} $\widetilde{X} = \widetilde{X} (I)$ as 
\begin{equation} \label{falka}
\widetilde{X} = \{f\in L^0(I): \widetilde{f}\in X\} ~~ {\rm with ~the ~norm } ~~ \|f\|_{\widetilde{X}}=\|\widetilde{f}\|_{X}.
\end{equation}
In particular, if $X = L_p, 1 < p \leq \infty$, we come to classical Ces\`aro spaces $Ces_p = CL_p$, which were investigated 
in \cite{AM08}, \cite{AM09}, \cite{AM13}, \cite{AM14}. Note, that the case $p = 1$ is not interesting here because it is easy to see that 
$Ces_1[0, \infty) = \{0\}$ and $Ces_1[0, 1] = L_1(\ln \frac{1}{t})$.
The space $Ces_{\infty}[0, 1]$ appeared already in 1948 and it is known as the Korenblyum-Krein-Levin space 
(see \cite{KKL48}, \cite{LZ66} and \cite{Wn99}).

It is clear that
\begin{equation} \label{embedding}
\widetilde{X} \overset{1}{\hookrightarrow } X, ~ {\rm and} ~ X \overset{A}{\hookrightarrow } CX ~{\rm provided} ~ 
C ~{\rm is ~ bounded ~ on} ~X ~ {\rm with} ~ A = \| C\|_{X \rightarrow X}.
\end{equation}
Moreover, if $X$ is a symmetric space on $I$, then for every $0 < a < b, b \in I$ we have
\begin{equation} \label{equality}
\| \chi_{[a, b]} \|_{\widetilde{X}} = \| \widetilde{\chi_{[a, b]}} \|_X = \| \chi_{[0, b]} \|_X = \varphi_X(b).
\end{equation}

In the sequence case the discrete Ces\`aro and Copson operators $C_d$ and $C_d^*$ are defined 
by 
$$
(C_d a)_n = \frac{1}{n} \sum_{k = 1}^n a_k\;\;\mbox{and}\;\; (C^*_d a)_n = \sum_{k = n}^{\infty} 
\frac{a_k}{k},\;\; n \in {\mathbb N},
$$  
and also the decreasing majorant $\widetilde{a} = (\widetilde{a_n})$ of a given sequence $a = (a_n)$ by
$$
\widetilde{a_n} = \sup_{k \in {\mathbb N}, \, k \geq n} |a_k|.
$$
Then the corresponding {\it abstract Ces\`aro sequence space} $CX$ and {\it abstract Tandori sequence space} $\widetilde{X}$ 
are defined similarly as in (\ref{Cesaro}) and (\ref{falka}). Again a lot is known about classical  Ces\`aro sequence spaces 
$ces_p: = Cl_p, 1 < p \leq \infty$ (cf. \cite{AM08}, \cite{AM13}, \cite{MPS07} and references given there).

Abstract Ces\`aro and Copson spaces were investigated in \cite{LM15a}, \cite{LM15b}, where the 
following results on their K\"othe duality were proved (cf. \cite[Theorems 3, 5 and 6]{LM15a}).
\vspace{1mm}

{\bf Theorem A.} {\it (i)  If $X$ is a Banach ideal space on $I = [0,\infty)$ such that the Ces\`aro operator $C$ and the dilation operator 
$\sigma_\tau$ for some $\tau \in (0, 1)$ are bounded on $X$, then 
\begin{equation} \label{ThmAi}
(CX)^{\prime} = \widetilde{X^{\prime}}.
\end{equation}
(ii) If $X$ is a symmetric space on $[0, 1]$ with the Fatou property such that the operators $C, C^*: X \rightarrow X$ are 
bounded, then
$$
(CX)^{\prime} = \widetilde{ X^{\prime}(w_1)}, ~ {\rm where} ~~ w_1(x) = \frac{1}{1-x}, ~x \in [0, 1).
$$
(iii) If $X$ is a Banach ideal sequence space such that the dilation operator $\sigma_{3}$ is bounded on $X^{\prime}$, then 
\begin{equation}  \label{ThmAiii}
(CX)^{\prime} = \widetilde{X^{\prime}}.
\end{equation}
Moreover, in the extreme case when $X = L_{\infty}(I)$ the above duality results hold with the equality of norms.}
\vspace{2mm}

The corresponding results on the K\"othe duality of the classical spaces $Ces_p(I)$ for $1 < p < \infty$ were proved in \cite{AM09} (see also \cite{KK12})
showing a surprising difference of them in the cases $I = [0, \infty)$ and $I = [0, 1]$. Much earlier, the identifications 
$(Ces_{\infty}[0, 1])^{\prime} \equiv \widetilde{L_1}[0, 1]$ and $(Ces_{\infty}^0[0, 1])^{\ast} \equiv {\tilde L_1}[0, 1]$
were obtained by Luxemburg and Zaanen \cite{LZ66} and by Tandori \cite{Ta55}, respectively. A simple proof of the latter results 
both for $I = [0, 1]$ and $I = [0, \infty)$ was given in \cite{LM15a}. 
Moreover, according to Theorem 7 from the above paper, if $w$ is a weight on $I$ such that $W(x) = \int_0^x w(t)\;dt < \infty$ for any $x \in I$, then 
setting $v(x) = {x}/{W(x)}$ we obtain
\begin{equation} \label{2.11}
(Ces_{\infty}(v))^{\prime}:= [C(L_{\infty}(v))]^{\prime} \equiv  \widetilde{L_1(w)}.
\end{equation} 
A close identification for weighted Ces\`aro sequence spaces follows from a result by Alexiewicz \cite{Al57}, who showed in 1957 
that for a weight $w = (w_n)$ with $w_n \geq 0, w_1 > 0$, we have 
\begin{equation} \label{Alex1}
\left(\widetilde{l_1(w)} \right)^{\prime} \equiv ces_{\infty}(v): = C(l_{\infty}(v)), ~
 {\rm where} ~ v(n) = \frac{n}{\sum_{k=1}^n w_k}.
\end{equation}
In particular, using the Fatou property of the space $\widetilde{l_1(w)},$ from \eqref{Alex1} we infer
\begin{equation} \label{Alex2}
\left( ces_{\infty}(v) \right)^{\prime} \equiv \left(\widetilde{l_1(w)} \right)^{\prime \prime} \equiv 
\widetilde{l_1(w)}.
\end{equation}

In \cite[Theorem 1(d)]{LM15a}), it was shown that if a Banach ideal space $X$ has the Fatou property, then the Ces\`aro and Tandori 
function spaces $CX$ and $\widetilde{X}$ also have it. Moreover, if a space $X$ is order continuous, then the Ces\`aro function 
space $CX$ is order continuous as well (cf. \cite[Lemma 1]{LM15b}). However, the Tandori function space $\widetilde{X}$ is never 
order continuous (cf. \cite[Theorem 1(e)]{LM15a}), which implies immediately that this space contains an isomorphic copy of $l_{\infty}$. 

Next, we repeatedly make use of the fact that every Ces\`aro function spaces $CX$ contains a complemented copy of $L_1[0,1]$. 

\begin{proposition}\label{Pro1n}
If $X$ is a Banach ideal function space on $I = [0,1]$ or $I = [0,\infty)$ such that the operator $C$ is bounded on $X$, then $CX$ contains 
a complemented copy of $L_1[0,1]$. Moreover, if $\chi_{[0, a]} \in X$ for $0 < a < 1$, then $\widetilde{X} \neq \{0\}$ and contains 
a complemented copy of $L_{\infty}[0, 1]$.
\end{proposition}
\proof
Suppose that $I = [0,1]$. Since $\supp X = I$ and the operator $C$ is bounded on $X$, then $\chi_{[a,1]} \in X$ 
for any $0 < a < 1$. In fact, let $f_0\in X$ with $f_0(x) > 0$ a.e. on $I$. Then, $f_0 \chi_{[0, a]} \in X$ and $\int_0^a f_0(t) \, dt = c > 0$. 
Therefore, from the estimate
\begin{eqnarray*}
C(f_0 \chi_{[0, a]})(x) 
&\geq& 
\frac{1}{x} \int_0^x f_0(t) \chi_{[0, a]}(t) dt \geq \frac{1}{x} \int_0^x f_0(t) \chi_{[0, a]}(t) dt \cdot \chi_{[a, 1]}(x) \\
&=&
\frac{1}{x} \int_0^a f_0(t) dt \cdot \chi_{[a, 1]}(x) \geq c \, \chi_{[a, 1]}(x), ~ 0 < x \leq 1,
\end{eqnarray*}
and from the boundedness of $C$ on $X$ it follows that $\chi_{[a,1]} \in X$. Now, for $0 < a < b < 1$ one has 
\begin{eqnarray*}
C( f \chi_{[a,b]})(x) 
&=&
\frac{1}{x}\int_a^x|f(t)| \,dt \, \chi_{[a,b]}(x) + \frac{1}{x}\int_a^b|f(t)| \,dt \, \chi_{[b,1]}(x) \\
&\leq&
\frac{1}{a}\int_a^b|f(t)| \,dt \,\Big[ \chi_{[a,b]} (x) + \chi_{[b,1]} (x) \Big] = \frac{1}{a} \, \|f \|_{L_1[a, b]} \cdot  \chi_{[a,1]} (x),
\end{eqnarray*}
and
\begin{eqnarray*}
C( f \chi_{[a,b]})(x) 
&=&
\frac{1}{x}\int_a^x|f(t)| \chi_{[a,b]} (t) \,dt \geq \frac{1}{x}\int_a^x|f(t)| \chi_{[a,b]} (t) \,dt \cdot \chi_{[b,1]} (x) \\
&\geq&
 \frac{1}{x}\int_a^b |f(t)| \,dt \cdot \chi_{[b,1]} (x) = \| f \|_{L_1[a, b]} \cdot \chi_{[b,1]} (x).
\end{eqnarray*}
Thus,
\begin{equation*}
d \, \|f \|_{L_1[a, b]} \leq \| f \chi_{[a,b]}\|_{CX} \leq \frac{D}{a}\|f \|_{L_1[a, b]},
\end{equation*}
where $d=\|\chi_{[b,1]}\|_X$ and $D=\|\chi_{[a,1]}\|_X$ are finite. Therefore, $CX|_{[a,b]}\simeq L_1[a,b]\simeq L_1[0,1]$ 
and, since the projection $P: f \mapsto f\chi_{[a,b]}$ is bounded, the first claim of the proposition is proved
if $I=[0,1]$.
The case $I = [0,\infty)$ can be treated in the same way, only the norm $\|\chi_{[b,1]}\|_X$ should be replaced with
 $\| \frac{1}{x} \,\chi_{[b,\infty)}(x)\|_X$. 

Regarding to the space $\widetilde{X}$ we note that under the conditions imposed on $X,$
by \eqref{equality}, we have 
$$ 
\widetilde{f \chi_{[a, b]}} \leq \| f  \|_{L_{\infty}[a, b]} \cdot \widetilde{\chi_{[a, b]}} = \| f \|_{L_{\infty}[a, b]} \cdot  \chi_{[0, b]}
$$
and conversely
$$ 
\widetilde{f \chi_{[a, b]}} \geq \widetilde{f \chi_{[a, b]}} \cdot \chi_{[0, a]} = \| f \|_{L_{\infty}[a, b]} \cdot \chi_{[0, a]},
$$
whence
$$
\|  \chi_{[0, a]} \|_X \, \| f \|_{L_{\infty}[a, b]} \leq  \| f \chi_{[a, b]} \|_{\widetilde{X}} \leq \|  \chi_{[0, b]} \|_X \,  \| f  \|_{L_{\infty}[a, b]}.
$$
Thus, the image of the same projection $Pf = f  \, \chi_{[a,b]}$ is isomorphic to $L_{\infty}[0, 1]$. Since $P$ is bounded, the proof 
is complete.
\endproof

\section{On the Schur and Dunford-Pettis properties of Ces\`aro and Tandori sequence spaces}

A Banach space $X$ is said to have the {\it Dunford-Pettis property} if, for all sequences $x_n \stackrel {w} \rightarrow 0$ 
in $X$ and $x_n^*\stackrel {w} \rightarrow 0$ in $X^*$, we have $\langle x^*_n, x_n \rangle \rightarrow 0$ as $n \rightarrow \infty$ 
or, equivalently, that any weakly compact operator $T: X \rightarrow Y$, where $Y$ is an arbitrary Banach space, is completely 
continuous (i.e., from $x_n \stackrel {w} \rightarrow 0$ it follows that $T(x_n)$ converges to $0$ in the norm of $Y$). Examples of 
spaces satisfying the Dunford-Pettis property are $l_1, c_0, l_{\infty}, L_1 (\mu), L_{\infty}(\mu)$ for every $\sigma$-finite measure 
$\mu$ and $C(K), C(K)^* = {\mathcal M}(K)$ for arbitrary compact Hausdorff set $K$ (cf. \cite[pp. 116-117]{AK06} and 
\cite[pp. 57-67]{Li04}). It is well-known that infinite dimensional reflexive spaces fail to have the Dunford-Pettis property. Moreover, 
if a dual space $X^*$ has the Dunford-Pettis property then so does $X$ (the reverse implication is not true) and complemented 
subspaces of spaces with the Dunford-Pettis property also have it.

Recall that a Banach space $X$ has the {\it Schur property} if for any sequence $x_n \stackrel {w} \rightarrow 0$ in $X$ we have  
$\| x_n \| \rightarrow 0$ as $n \rightarrow \infty$ or, equivalently, if every weakly compact operator from $X$ to arbitrary Banach space 
$Y$ is compact. Of course, spaces with the Schur property have also the Dunford-Pettis property. Even though it has been known 
since the famous Banach's book \cite[pp. 137-139]{Ba32} was published that the space $l_1$ has the Schur property, only a few 
natural infinite dimensional spaces enjoying it were found.

A survey of results related to the Dunford-Pettis property and the Schur property one can find in \cite{Di80} and \cite{Wn93}, 
respectively (see also \cite{CG94}).

We start with proving the Schur property in the space $\widetilde{l_1}$. Note that $\widetilde{l_1}$ is not isomorphic to $l_1$. In fact, 
$\{\frac{1}{n} e_n\}$ is a normalized unconditional basis in $\widetilde{l_1}$. On the other hand, $l_1$ has unique unconditional 
structure, i.e., each normalized unconditional basis in $l_1$ is equivalent to $\{e_n\}$ (cf. \cite[Theorem 2.b.9]{LT77}). Therefore, 
if we assume that $\widetilde{l_1}$ is isomorphic to $l_1$, then $\{\frac{1}{n}e_n\}$ would be equivalent to $\{e_n\}$. But it is not 
the case since we have both $\|\sum_{n=1}^k\frac{1}{n}e_n\|_{\widetilde{l_1}}\approx \ln k$ and $\|\sum_{n=1}^ke_{n}\|_{l_1}= k$, 
$k\in\mathbb{N}$. 

\begin{theorem}\label{schur}
The space $\widetilde{l_1}$ has the Schur property.
\end{theorem}
\proof
First, using (\ref{Alex1}) and the fact that $\widetilde{l_1}$ has an order continuous norm (cf. \cite{LM15b}) we obtain
$(\widetilde{l_1})^* = (\widetilde{l_1})^{\prime} = ces_{\infty}$. Now, let $\| x^{(n)} \|_{\widetilde{l_1}} \leq 1$ with 
$x^{(n)}\rightarrow 0$ weakly in $\widetilde{l_1}$ as $n \rightarrow \infty$. By the Banach-Alaoglu theorem 
the closed unit ball $B$ in $ces_{\infty}$ is $w^*$-compact and metrizable, so, in particular, it is a $w^*$-complete 
metric space. For any $\varepsilon > 0$ we put
\begin{equation*}
B_m=\bigcap_{n\geq m} \{f\in B: |\langle f, x^{(n)}\rangle|\leq \varepsilon\}.
\end{equation*}
Then the sets $B_m$ are $w^*$-closed, $B_1\subset B_2  \subset \dots$ and $B=\bigcup_{m=1}^\infty B_m$.
Thus, by the Baire theorem, there are $N, m_1\in \mathbb{N}$, $g = (g_k) \in B_{m_1}$ and $\delta>0$ such that 
\begin{equation*}
U: = \{f = (f_k) \in B: |f_k - g_k| < \delta, 1\leq k\leq N\} \subset B_{m_1}.
\end{equation*}
Consequently, $U\subset B_{m}$ for each $m\geq m_1$. Fix $N_1>N$ such that 
\begin{equation} \label{3.1}
\frac{\sum _{k=1}^{N}|g_k|}{N_1}<\varepsilon.
\end{equation}
Clearly, the weak convergence of $\{x^{(n)}\}$ implies the coordinate convergence, so that there 
is $m_2 \in \mathbb N$ such that for $n \geq m_2$
\begin{equation} \label{3.2}
\|x^{(n)}\chi_{[1,N_1)}\|_{\widetilde{l_1}}\leq \varepsilon.
\end{equation}
For every $n \in \mathbb N$ there is $f^{(n)}\in B$ such that 
\begin{equation*}
\|x^{(n)}\chi_{[N_1,\infty)}\|_{\widetilde{l_1}}=\langle f^{(n)},x^{(n)}\chi_{[N_1,\infty)}\rangle.
\end{equation*}
Without loss of generality, we may assume that $\supp f^{(n)}\subset [N_1,\infty)$. Setting
\begin{equation*}
g^{(n)} = g\chi_{[1,N)}+ (1-\varepsilon)f^{(n)},
\end{equation*}
we will show that $g^{(n)}\in U$ for all $n \in \mathbb N$. Since $g^{(n)}_k=g_k$ for each $1\leq k\leq N$, it is enough to 
check only that $\|g^{(n)}\|_{ces_{\infty}} \leq 1$. We have
\begin{equation*}
C|g^{(n)}|(j)=\left\{
\begin{array}{ccc}
 C|g|(j), & & ~ {\rm if} ~ j < N, \\ 
\frac{1}{j}\sum_{k=1}^N|g_k|, &  & ~ {\rm if} ~ N \leq j < N_1 \\ 
\frac{1}{j}(\sum_{k=1}^N|g_k|+(1-\varepsilon)\sum_{k=N_1}^j |f^{(n)}_k|), &  & ~ {\rm if} ~ j\ge N_1.
\end{array}
\right.
\end{equation*}
Consequently, from (\ref{3.1}) it follows that  $C|g^{(n)}|(j)\leq 1$ for each $j \in \mathbb N$, i.e. $g^{(n)}\in U$ and therefore 
$g^{(n)} \in B_m$ if $m > m_1$. Finally, applying (\ref{3.2}), for $n \geq m_0: =\max\{m_1,m_2\}$ we get
\begin{eqnarray*}
\|x^{(n)}\|_{\widetilde{l_1}}&\leq& \|x^{(n)}\chi_{[1, N_1)}\|_{\widetilde{l_1}}+\|x^{(n)}\chi_{[N_1,\infty)}\|_{\widetilde{l_1}} \leq \varepsilon 
+ |\sum_{k=N_1}^{\infty} x^{(n)}_kf^{(n)}_k|\\ 
&=&
\varepsilon +\Big|\frac{1}{1-\varepsilon}\Big( (1 - \varepsilon) \sum_{k=N_1}^{\infty} x^{(n)}_kf^{(n)}_k +\sum_{k=1}^{N} x^{(n)}_kg_k \Big) 
- \frac{1}{1-\varepsilon}\sum_{k=1}^{N} x^{(n)}_kg_k \Big|\\ 
&\leq &
\varepsilon + \frac{1}{1-\varepsilon} |\sum_{k=1}^{\infty} x^{(n)}_k g^{(n)}_k| + \frac{1}{1-\varepsilon}\|x^{(n)}\chi_{[1,N)}\|_{\widetilde{l^1}}\|g\|_{ces_{\infty}} \\
&=&
\varepsilon + \frac{1}{1-\varepsilon} | \langle g^n, x^n \rangle | + \frac{1}{1-\varepsilon}\|x^{(n)}\chi_{[1,N)}\|_{\widetilde{l^1}}\|g\|_{ces_{\infty}} 
\leq \varepsilon + \frac{2 \varepsilon}{1-\varepsilon},
\end{eqnarray*}
which shows that $\lim_{n \rightarrow \infty} \|x^{(n)}\|_{\widetilde{l_1}} = 0$, as desired.
\endproof

\begin{corollary} \label{cesDP}
The space $ces_{\infty}^0$ has the Dunford-Pettis property. 
\end{corollary}
\proof
From (\ref{Alex2}) we have $(ces_{\infty}^0)^* = (ces_{\infty}^0)^{\prime} = \widetilde{l_1}$ and by the fact that a Banach 
space has the Dunford-Pettis property whenever its dual space has it, the result follows from Theorem \ref{schur}.
\endproof

Although the spaces $\widetilde{l_1}$ and $l_1$ are not isomorphic, $l_1$ is isomorphic to a subspace of $\widetilde{l_1}$ 
and so $\widetilde{l_1}$ can be treated as an extension of $l_1$ with preserving the Schur property.

\begin{theorem}\label{Thm2}
 The basic sequence $\{2^{-i}e_{2^i}\}_{i=0}^\infty$ is equivalent in the space $\widetilde{l_1}$ to the canonical
$l_1$-basis.
\end{theorem}
\proof We prove that for all $n\in\mathbb{N}$ and $c_i \geq 0, i = 0, 1,\dots, n$,
$$
\frac{1}{36} \, \sum_{i=0}^n c_i\le\Big\|\sum_{i=0}^n c_i2^{-i}e_{2^i}\Big\|_{\widetilde{l_1}} \leq \,\sum_{i=0}^n c_i.$$
Since $\|e_m\|_{\widetilde{l_1}}=m,$ for every $m\in\mathbb{N}$, the right-hand side inequality is obvious. Thus, it is enough 
to check only the opposite inequality.

We choose the subset of indices from the set $\{0, 1, \ldots, n\}$ according to the following procedure.
Let $k_0=n.$ If $2^{-k_0}c_{k_0}\ge 2^{-k}c_k$ for each $k=0,\dots,n-1,$
we put $I=\{k_0\}$ and finish. Otherwise, we define
$$
k_1:=\max\{k=0,\dots,n-1:\,2^{-k_0}c_{k_0}< 2^{-k}c_k\}.$$
Similarly, if $2^{-k_1}c_{k_1}\ge 2^{-k}c_k$ for every $k=0,\dots,k_1-1,$
we set $I=\{k_0,k_1\}$ and finish. Otherwise, let
$$
k_2:=\max\{k=0,\dots,k_1-1:\,2^{-k_1}c_{k_1}< 2^{-k}c_k\}.$$
Proceeding in the same way, we construct the set
$$
I=\{k_i\}_{i=0}^l \subset \{0,1,\dots,n\},\;\;n=k_0>k_1>\dots>k_l\ge 0$$
satisfying the following conditions:
\begin{equation}\label{33}
 2^{-k_i}c_{k_i}\ge 2^{-k}c_k\;\;\mbox{if}\;\;k_{i+1}<k\le k_i,
\end{equation}
\begin{equation}\label{34}
 2^{-k_l}c_{k_l}\ge 2^{-k}c_k\;\;\mbox{if}\;\;0\le k\le k_l
\end{equation}
and
\begin{equation}\label{35}
 2^{-k_i}c_{k_i}< 2^{-k_{i+1}}c_{k_{i+1}}\;\;\mbox{for all}\;\;i=0,1,\dots,l-1.
\end{equation}
Observe that from \eqref{33} and \eqref{34} it follows that
$$
c_k\le 2^{k-k_i}c_{k_i}\;\;\mbox{if}\;\;k_{i+1}<k\le k_i,i=0,1,\dots,l-1,$$
and
$$
 c_{k}\le 2^{k-k_l}c_{k_l}\;\;\mbox{if}\;\;0\le k\le k_l.$$
Hence,
\begin{equation}\label{36}
\sum_{k=k_{i+1}}^{k_i} c_k\le 2c_{k_i},\;\;i=0,1,\dots,l-1,\;\;\mbox{and}\;\;\sum_{k=0}^{k_l} c_k\le 2c_{k_l}.
\end{equation}
Further, we define the set $I_1\subset I$ as follows:
$$
I_1= \Big\{k_i\in I,i=1,\dots,l-1:\,c_{k_i}<\frac32 c_{k_{i+1}}\Big\}.$$
Then, $I_1=\{k_{i_j}\}_{j=1}^s,$ where $n\ge k_{i_1}>k_{i_2}>\dots >k_{i_s},$ $0\le i_1<i_2<\dots<i_s<l.$
It is easy to see that $c_{k_{i+1}}\le\frac23 c_{k_i}$ whenever $i\ne i_j,$ $j=1,\dots,s.$ 
Therefore, if $j=1,\dots,s,$
 \begin{equation}\label{37}
\sum_{i=i_j}^{i_{j+1}-1} c_{k_i}\le c_{k_{i_j}}+\sum_{i=i_j+1}^{i_{j+1}-1} \Big(\frac23\Big)^{i-i_j-1} c_{k_{i_j+1}}\le 
\frac32 c_{k_{i_j+1}}+ 3c_{k_{i_j+1}}= \frac92 c_{k_{i_j+1}},\;\;j=1,\dots,s
\end{equation}
(in what follows we set $i_{s+1}=l+1$).
Moreover, if $n\not\in I_1,$ we have
 \begin{equation}\label{38}
\sum_{i=0}^{i_1-1} c_{k_i}\le 3 \, c_{k_{0}}=3 c_{n}.
\end{equation}
By the definition of norm in $\widetilde{l_1}$  and from inequalities \eqref{33}, \eqref{34} and \eqref{35} we obtain
 \begin{eqnarray*}
\Big\|\sum_{i=0}^n c_i2^{-i}e_{2^i}\Big\|_{\widetilde{l_1}} &=&  2^{-k_0}c_{k_0}\cdot 2^{k_0}+(2^{-k_1}c_{k_1}-2^{-k_0}c_{k_0}) \, 2^{k_1}+\dots\\
&+& (2^{-k_{i+1}}c_{k_{i+1}}-2^{-k_i}c_{k_i}) \, 2^{k_{i+1}}+\dots +(2^{-k_l}c_{k_l}-2^{-k_{l-1}}c_{k_{l-1}}) \, 2^{k_l}\\
&=& c_{k_0}+ (c_{k_{1}}-2^{k_1-k_0}c_{k_0})+\dots +(c_{k_{i+1}}-2^{k_{i+1}-k_{i}}c_{k_{i}})+\dots \\
&\dots& + \, (c_{k_{l}}-2^{k_{l}-k_{l-1}}c_{k_{l-1}})\\
&\ge& c_n+\sum_{j=1}^s (c_{k_{i_j+1}}-2^{k_{i_j+1}-k_{i_j}}c_{k_{i_j}}).
 \end{eqnarray*}
If $k_i\in I_1,$ then, by the definition of $I_1,$ we have
$$
(c_{k_{i+1}}-2^{k_{i+1}-k_{i}}c_{k_{i}})\ge c_{k_{i+1}}\Big(1-\frac32\cdot 2^{k_{i+1}-k_{i}}\Big)\ge \frac14 c_{k_{i+1}},$$
because of $k_i>k_{i+1}.$ Therefore,
\begin{equation}\label{39}
 \Big\|\sum_{i=0}^n c_i2^{-i}e_{2^i}\Big\|_{\widetilde{l_1}}\ge c_n+\frac14 \sum_{j=1}^s c_{k_{i_j+1}}.
\end{equation}
On the other hand, by \eqref{36}--\eqref{38}
\begin{eqnarray*}
 \sum_{i=0}^n c_i &=& \sum_{i=0}^{l-1}\sum_{k=k_{i+1}+1}^{k_i} c_k+ \sum_{k=0}^{k_l} c_k\le 2 \sum_{i=0}^l c_{k_i}\\
&\le& 2\Big(\sum_{i=0}^{i_{1}-1} c_{k_i} + \sum_{j=1}^{s}\sum_{i=i_j}^{i_{j+1}-1} c_{k_i}\Big)\le 9  \sum_{j=1}^{s} c_{k_{i_j+1}}+6c_n.
\end{eqnarray*}
Combining the latter inequality with \eqref{39}, we infer 
$$
\sum_{i=0}^n c_i \le 36\cdot \Big\|\sum_{i=0}^n c_i2^{-i}e_{2^i}\Big\|_{\widetilde{l_1}},$$
and the proof is complete.
\endproof

Since the space $\widetilde{l_1}$ is order continuous, then from (\ref{Alex1}) it follows $(\widetilde{l_1})^*= (\widetilde{l_1})^{\prime} = ces_{\infty}$. 
Therefore, taking into account that the space $ces_{\infty}$ has the Fatou property, we obtain 

\begin{corollary} \label{Cor2}
The basic sequence $\{2^{i}e_{2^i}\}_{i=0}^\infty$ is equivalent in the space $ces_{\infty}$ to the canonical
$c_0$-basis and $l_{\infty}$ is embedded into $ces_{\infty}$. 
\end{corollary}

\begin{corollary}  \label{Cor3new}
The space $\widetilde{l_1}$ is isomorphic to the space $(\bigoplus_{n=0}^{\infty} l_\infty^{2^n})_{l_1}$. Therefore, the 
space $ces_{\infty}$ is isomorphic to the space $(\bigoplus_{n=0}^{\infty} l_1^{2^n})_{l_{\infty}}$ and contains a complemented 
subspace isomorphic to $L_1[0,1]$.
\end{corollary}
\proof Let us define the linear operator $T$ from $\widetilde{l_1}$ to $(\bigoplus_{n=0}^{\infty} l_\infty^{2^n})_{l_1}$ 
as follows: if $c = (c_k)_{k=1}^{\infty} \in \widetilde{l_1}$, then $Tc = (d^{(n)})_{n = 0}^{\infty}$, where 
$d^{(n)} = (d_j^{(n)})_{j = 1}^{2^n}, d_j^{(n)}: = (j - 1 + 2^n) \cdot c_{j - 1 + 2^n}, j = 1, 2, \ldots, 2^n$.

Assuming that $c_k \geq 0, k = 1, 2, \ldots$, by Theorem \ref{Thm2}, we obtain
$$
\| c \|_{\widetilde{l_1}} = \| \sum_{k=0}^{\infty} c_k \, e_k \|_{\widetilde{l_1}} 
= \Big \| \sum_{n=0}^{\infty} \sum_{j = 1}^{2^n}  \frac{d_j^{(n)}}{j - 1 + 2^n} \, e_{2^n + j - 1} \Big \|_{\widetilde{l_1}} 
$$
$$
\leq \Big \| \sum_{n=0}^{\infty} \max_{j = 1, \ldots, 2^n} d_j^{(n)} \cdot 2^{- n} \, e_{2^{n + 1}} \Big \|_{\widetilde{l_1}} 
\leq 2 \, \sum_{n=0}^{\infty} \max_{j = 1, \ldots, 2^n} d_j^{(n)} = 2 \, \| Tc \|_{(\bigoplus_{n=0}^{\infty} l_\infty^{2^n})_{l_1}}.
$$
On the other hand, by the definition of the norm in $\widetilde{l_1}$ and Theorem \ref{Thm2}, we have
$$
\| c \|_{\widetilde{l_1}} = \Big \| \sum_{n=0}^{\infty} \sum_{j = 1}^{2^n}  \frac{d_j^{(n)}}{j - 1 + 2^n} \, e_{2^n + j - 1} \Big \|_{\widetilde{l_1}} 
\geq \frac{1}{2} \,  \Big \| \sum_{n=0}^{\infty} \max_{j = 1, \ldots, 2^n} d_j^{(n)} \cdot 2^{- n} \, e_{2^n} \Big \|_{\widetilde{l_1}} 
$$
$$
\geq \frac{1}{72} \, \sum_{n=0}^{\infty} \max_{j = 1, \ldots, 2^n} d_j^{(n)} = \frac{1}{72} \, \| Tc \|_{(\bigoplus_{n=0}^{\infty} l_\infty^{2^n})_{l_1}},
$$
and therefore $T$ is an isomorphism from $\widetilde{l_1}$ onto $(\bigoplus_{n=0}^{\infty} l_\infty^{2^n})_{l_1}$. Since
$(\widetilde{l_1})^{\prime} = ces_{\infty}$, by duality, we deduce that $ces_{\infty}$ is isomorphic to 
the space $(\bigoplus_{n=0}^{\infty} l_1^{2^n})_{l_{\infty}}$.

To get the last result of corollary, we note that $(\bigoplus_{n=0}^{\infty} l_\infty^n)_{l_1}$ is a complemented subspace of 
$(\bigoplus_{n=0}^{\infty} l_\infty^{2^n})_{l_1}$ and hence of $\widetilde{l_1}$. Thus, applying the Hagler-Stegall theorem 
(see \cite[Theorem 1]{HS73}) we conclude that the dual space, i.e., $ces_{\infty}$ contains a complemented subspace isomorphic 
to $L_1[0, 1]$.
\endproof

\begin{remark} \label{Rem4}
In \cite[p. 19]{B81}, Bourgain proved that arbitrary $l_1$-sum of finite-dimensional  Banach spaces has the Schur property (see 
also \cite[pp. 60-61]{CI90} for a simpler proof). Hence, from Corollary \ref{Cor3new} we can infer that $\widetilde{l_1}$ has the 
Schur property and thereby we get another proof of Theorem \ref{schur}.
\end{remark}

From Corollary \ref{Cor3new} and the fact that $l_\infty$ is a prime space (cf. \cite[Thm 5.6.5]{AK06} and \cite[Thm 2.a.7]{LT77}) it follows

\begin{corollary} \label{Cor5}
The spaces $ces_{\infty}$ and $l_\infty$ are not isomorphic.
\end{corollary}

The fact that the space $ces_{\infty}$ contains a complemented subspace isomorphic to $L_1[0, 1]$ will be next a crucial tool in proving 
the existence of an isomorphism between $Ces_{\infty}$-spaces of functions and sequences. So, it is worth to give its direct proof without 
referring to the general theorem of Hagler-Stegall (\cite[Theorem 1]{HS73}), especially, because the following proof, we hope, is 
interesting in its own.

\begin{theorem}\label{Thm3}
The space $ces_{\infty}$ contains a complemented subspace isomorphic to $L_1[0,1]$.
\end{theorem}
\proof Thanks to Corollary \ref{Cor3new}, it is sufficient to prove that in the space $(\bigoplus_{n=0}^{\infty}l_1^{2^{n}})_{l_\infty}$
there is a complemented subspace isomorphic to $L_1[0,1].$
Denote the collection of dyadic intervals of $[0, 1]$ by
$$
B^n_k = [\frac{k-1}{2^{n}},\frac{k}{2^{n}}), {\rm \ where \ } k=1,\dots,2^{n} {\rm \ and\ } n = 0, 1, 2, \dots
$$
and define a sequence of operators $H_n: L^1\rightarrow l_1^{2^{n}}$  by
$$
H_n: f\mapsto \Big\{\int_{B_n^k}f(t)dt\Big\}_{k=1}^{2^n}. 
$$
Then $\|H_n\|=1$ for each $n$. Moreover, put $H: f \mapsto \oplus_{n=0}^{\infty} \, H_nf$. Then $H$ maps $L_1[0,1]$ into 
the space $(\bigoplus_{n=0}^{\infty}l_1^{2^{n}})_{l_\infty}$ with $\|H\|=1$. Moreover, let us show that $H$ is an isometry between 
the spaces $L_1[0, 1]$ and $H(L_1[0,1])$. In fact, denoting by $\{h_k\}$ the Haar basis, for a given $f \in L^1[0,1]$ and any 
$\varepsilon > 0$ one can find a function $g = \sum_{k=1}^N a_k h_k$ such that $\| g  - f \| \leq \varepsilon$. Then, for $n$ large 
enough, there holds $\|H_ng\|=\|g\|$, which, in consequence, gives $\|H_nf\|\geq \|H_ng\| - \varepsilon\geq \|f\| - 2 \varepsilon$ 
and proves our claim. 

To see that $H(L_1[0,1])$ is complemented in $(\bigoplus_{n=0}^{\infty}l_1^{2^{n}})_{l_\infty}$ a little more work is required. 
For a given $n$ we set $T_n: (\bigoplus_{n=0}^{\infty}l_1^{2^{n}})_{l_\infty}\rightarrow L_1[0,1]$ by 
$$
T_n:x\mapsto \sum_{k=1}^{2^n}2^nx_n^k\chi_{B_n^k},
$$
where $x_n=(x_n^k)_{k=1}^{2^n}$ and $x=\oplus_{n=0}^{\infty}x_n$.
Of course, $\|T_n\|=1$ for each $n$. Let $\eta$ be a free ultrafilter. Then for a given $x\in (\bigoplus_{n=0}^{\infty}l_1^{2^{n}})_{l_\infty}$ 
we define the functional $R_x$ on $C[0,1]$ by the formula
$$
R_x(g) = \lim_{\eta} \langle T_nx, g \rangle \ {\rm for} \ g\in C[0,1].
$$
Since $\|T_n\|=1, n \in \mathbb N$ we get $\|R_x\|=1$. 

Recalling that the space $C[0,1]^*$ consists of all regular Borel measures on $[0, 1]$ with finite variation, denote by $Q$ the Lebesgue  
projection which maps any such measure into its absolutely continuous part. Now, one can verify that $P: x \mapsto H(Q(R_x))$ is the 
required projection from $(\bigoplus_{n=0}^{\infty}l_1^{2^{n}})_{l_\infty}$ onto $H(L_1[0,1])$.
In fact, since  $\{T_n(Hf)\}_{n=0}^{\infty}$ is a uniformly integrable martingale associated to a function $f\in L_1[0,1]$, there holds 
$T_n(Hf) \rightarrow f$ in $L_1[0,1]$-norm. In consequence,  
$$
R_{Hf}(g) = \lim_{\eta} \langle T_nHf,g \rangle = \lim_{n\rightarrow \infty} \langle T_nHf, g\rangle = \langle f, g\rangle ~ {\rm for} \ g\in C[0,1].
$$
Therefore, $Q(R_{Hf}) = f$, which proves our claim.
\endproof

Now, we investigate the conditions under which Ces\`aro and Tandori sequence spaces have and do not have the Dunford-Pettis 
property.

Observe that under the assumption of nontriviality of indices of a function $\varphi$, that is, when $0 < p_{\varphi} \leq q_{\varphi} < 1$, 
in the case $[0, \infty)$ we have $C\Lambda_{\varphi} = L_1(\varphi (t)/t)$ with equivalent norms (see \cite[Theorem 4.4]{DS07} and 
\cite[Theorem 8]{LM15a}) and hence, by Theorem A(i), the corresponding Tandori space $\widetilde{M_{\varphi}} = L_{\infty}(\varphi)$, 
so they both have the Dunford-Pettis property. An inspection of the proof of Theorem~8 from \cite{LM15a} combined with duality
(see Theorem A(iii)) shows that for the respective sequence spaces the following result holds.

\begin{theorem}\label{Thm4a} 
Let $\varphi$ be an increasing concave function on $[0, \infty)$.

(i) If $p_{\varphi}^\infty>0$, then $C\lambda_{\varphi}=l_1(\varphi(n)/n)$ and therefore has 
the Dunford-Pettis property. 

(ii) If $q_{\varphi}^\infty < 1$, then $\widetilde{m_{\varphi}}=l_\infty(\varphi(n))$
and therefore has the Dunford-Pettis property. 
\end{theorem}

In the proof of a similar result related to the spaces $\widetilde{\lambda_{\varphi}}$ and $Cm_{\varphi}$
we will make use of a suitable isomorphic description of
these spaces and the well-known Bourgain's results mentioned in the Introduction (see \cite{Bo81}).

\begin{theorem}\label{Thm5a} 
For arbitrary  increasing concave function $\varphi$ on $[0, \infty)$ the spaces $\widetilde{\lambda_{\varphi}}$ and 
$Cm_{\varphi}$ have the Dunford-Pettis property. 
\end{theorem}

\begin{proof}
 At first, in the case when $\lim_{t\to\infty}\varphi(t)<\infty$ we have ${\lambda_{\varphi}}=l_\infty$, whence 
 $\widetilde{\lambda_{\varphi}}=l_\infty,$ and the result follows. So, let $\lim_{t\to\infty}\varphi(t)=\infty$.  
Moreover, the function $\varphi(t)$ is strictly increasing and, without loss of generality, we can assume that  $\varphi(1)=1.$
Let us define the increasing sequence $\{n_k\}_{k=1}^\infty$, where $n_1=1$,
as follows
$$
n_{k+1}:=\sup\{i>n_k:\,\varphi(i)-\varphi(n_k)\le 2^k\},\;\;k=1,2,\dots$$
Then, since $\varphi(n_{k+1}+1)-\varphi(n_k)>2^k$ and, by subadditivity of $\varphi,$
$$
\varphi(n_{k+1}+1)-\varphi(n_{k+1})\le \varphi(1)=1,$$ we have
\begin{equation} \label{3a.1}
2^{k-1}\le \varphi(n_{k+1})-\varphi(n_k)\le 2^k,\;\;k=1,2,\dots
\end{equation}
Therefore, by \eqref{3a.1} and \cite[Proposition 2.1]{GHS96} (see also \cite[Lemma 3.2]{GP03}), 
for every $x=(x_n)_{n=1}^\infty \in \widetilde{\lambda_{\varphi}}$ we have

\begin{eqnarray*}
\| x \|_{\widetilde{\lambda_{\varphi}}}
&=&
\sum_{n=1}^\infty  \widetilde{x_n} (\varphi(n+1)-\varphi(n))=
\sum_{k=1}^\infty \sum_{n=n_k}^{n_{k+1}-1} \widetilde{x_n} (\varphi(n+1)-\varphi(n))\\
&\le& \sum_{k=1}^\infty \widetilde{x_{n_k}} (\varphi(n_{k+1})-\varphi(n_k))=
\sum_{k=1}^\infty 2^k\sup_{j\ge k}\max_{n_j\le i\le n_{j+1}}|x_{i}|\\
&\le& C\sum_{k=1}^\infty 2^k\max_{n_k\le i\le n_{k+1}}|x_{i}|=C\| x \|_{\oplus}
\end{eqnarray*}
with some constant $C>0,$ where 
$$
 \| x \|_{\oplus}: = \sum_{k=1}^\infty 2^k\max_{n_k\le i\le n_{k+1}}|x_{i}|.
$$
Conversely, again, by \eqref{3a.1},
\begin{eqnarray*}
\| x \|_{\oplus} 
&=&
2\max_{1\le i\le n_{2}}|x_{i}|+\sum_{k=2}^\infty 2^k\max_{n_k\le i\le n_{k+1}}|x_{i}|\le
2\|x\|_{l_\infty}+4\sum_{k=2}^\infty \widetilde{x_{n_k}} (\varphi(n_{k})-\varphi(n_{k-1}))\\
&\le& 2\|x\|_{l_\infty}+4\sum_{k=2}^\infty \sum_{n=n_{k-1}}^{n_{k}-1} \widetilde{x_n} (\varphi(n+1)-\varphi(n)) \le
\left(\frac{2}{\varphi(2)-1}+4\right)\|x\|_{\widetilde{\lambda_{\varphi}}}.
\end{eqnarray*}
These inequalities show that $\widetilde{\lambda_{\varphi}}$ is isomorphic to the space
$(\bigoplus_{k \in \mathbb N} l_{\infty}^{m_k})_{l_1},$ where
$m_k=n_{k+1}-n_k,$ $k \in \mathbb N.$ 
Hence, applying \cite[Corollary 7]{Bo81}, we obtain that the space
$\widetilde{\lambda_{\varphi}}$ has the Dunford-Pettis property.

Regarding $Cm_{\varphi}$ we note, firstly, that in the case when
$\lim_{t\to\infty}\varphi(t)/t>0$ the latter space coincides with $l_1$
and hence has the Dunford-Pettis property. If $\lim_{t\to\infty}\varphi(t)/t=0$,
then from Theorem A(iii) and the first part of the proof it follows that
$$
(Cm_{\varphi})^{\prime} = \widetilde{\lambda_{\psi}} \, \simeq \, (\bigoplus_{k \in \mathbb N} l_{\infty}^{m_k})_{l_1},
$$
where $\psi(t) = t/{\varphi(t)}$.
Combining this with the fact that $Cm_{\varphi}$ has the Fatou property, we infer
$$
Cm_{\varphi} \, \simeq  \, [ (\bigoplus_{k \in \mathbb N} l_{\infty}^{m_k})_{l_1}]^{\prime} = 
(\bigoplus_{k \in \mathbb N} l_1^{m_k})_{l_{\infty}},
$$
Hence, from Bourgain's result \cite[Theorem 1]{Bo81} it follows 
that $Cm_{\varphi}$ has the Dunford-Pettis property. 
\end{proof}

\begin{corollary} \label{Cor5a}
For any increasing concave function $\varphi$ on $[0, \infty)$ the space $C(m_{\varphi}^0)$ has the Dunford-Pettis property. 
\end{corollary}

\begin{proof}
Since the space $C(m_{\varphi}^0)$ is order continuous, by Theorem A(iii), we have
$$
[C(m_{\varphi}^0)]^*=[C(m_{\varphi}^0)]^{\prime}= \widetilde{\lambda_{\psi}},$$
where $\psi(t) = t/{\varphi(t)}$. Now desired result follows from the preceding theorem. 
\end{proof}

Now, we show that in the case of reflexive spaces the situation is completely different.

\begin{theorem}\label{Thm4} 
Let $X$ be a reflexive symmetric sequence space.

(i) If the operator $C_d$ is bounded on $X$, then $CX$ does not have the Dunford-Pettis property. 

(ii) If the operator $C_d$ is bounded on $X^{\prime}$, then $\widetilde{X}$ does not have the Dunford-Pettis property. 
\end{theorem}
\proof (i) Since $X$ is reflexive it follows that $C_d: CX\rightarrow X$ is a weakly compact operator. Therefore, it is sufficient 
to show that $C_d$ is not a Dunford-Pettis operator. 

Consider the sequence $x_n = \varphi_{X'}(n) \,e_n$, where $\{e_n\}$ is the canonical basis in $X$ and $\varphi_{X'}$ is 
the fundamental function of $X'$. Let us show that $x_n\rightarrow 0$ weakly in $CX.$ 

Since the space $X$ has an order continuous norm, then $CX$ has also order continuous norm and, by Theorem A(iii),  
we obtain $(CX)^* = (CX)' = \widetilde{X'}$. Therefore, by the definition of $\widetilde{X'}$, it is sufficient to prove that
\begin{equation} \label{3.10}
\langle y, x_{n}\rangle = \varphi_{X'}(n) \, y_{n}\to 0\;\;\mbox{as}\;\;n\to\infty
\end{equation}
for each non-increasing positive sequence $y=(y_n) \in X'$. Observe that $X'\subset m_{\varphi_{X'}}$, where $m_{\varphi_{X'}}$ 
is the Marcinkiewicz space with the fundamental function $\varphi_{X'}$. Moreover, by reflexivity of $X$, $X' = (X')^0,$ and thus  
$$
X'\subset m_{\varphi_{X'}}^0 = \{(z_n):\lim_{n\rightarrow \infty}\varphi_{X'}(n) \,z^*_n=0\}.
$$
Clearly, this embedding implies \eqref{3.10} and hence $x_n \rightarrow 0$ weakly in $CX$.

On the other hand, $C_de_n=\sum_{k=n}^{\infty}e_k/k$ for every $n \in \mathbb N$. Since $X$ is a symmetric space, we have
\begin{eqnarray*}
\|C_de_n\|_X 
&=&
\Big\|\sum_{k=1}^{\infty}\frac{e_k}{n+k-1}\Big\|_X \geq \Big\|\sum_{k=1}^n \frac{e_k}{n+k-1}\Big\|_X \\
&\geq &
\frac{1}{2n} \, \| \sum_{k=1}^n e_k \|_X = \frac{\varphi_X(n)}{2n}=\frac{1}{2 \, \varphi_{X'}(n)}.
\end{eqnarray*}
Hence, $\|C_dx_n\|_X\ge 1/2$ for all $n,$ and the proof is completed.

(ii) Since $X^{\prime}$ is reflexive and $C_d$ is bounded on $X^{\prime}$, from Theorem \ref{Thm4}(i) it follows that
$C(X^{\prime})$ does not have the Dunford-Pettis property. By duality, Theorem A(iii) and the fact that $C (X^{\prime})$ 
is order continuous, we obtain $\widetilde{X} = (CX^{\prime})^{\prime} =  (CX^{\prime})^*$. Thus, $\widetilde{X}$ is the dual 
space to a space without the Dunford-Pettis property. So, $\widetilde{X}$ also fails to have it.
\endproof

\section{On the Dunford-Pettis property of Ces\`aro and Tandori function spaces}

As was mentioned above, under the assumption of nontriviality of indices of a function $\varphi,$ 
the spaces $C\Lambda_{\varphi}$ and $\widetilde{M_{\varphi}}$ are some weighted
$L_1$- and $L_\infty$-spaces, respectively, and so they both have the Dunford-Pettis property
(see \cite[Theorem 4.4]{DS07} and \cite[Theorem 8]{LM15a}).
Similarly, as in Theorem \ref{Thm4a}, we are able to prove the latter property also
for their counterparts, $\widetilde{\Lambda_{\varphi}}$ and $CM_{\varphi}$.

\begin{theorem} \label{Thm5}
Let $\varphi$ be an increasing concave function on $[0, \infty)$.

(i) The space $\widetilde{\Lambda_{\varphi}}[0, \infty)$ is isomorphic to the space $(\bigoplus_{n \in \mathbb N} L_{\infty}[0, 1])_{l_1}$ 
and has the Dunford-Pettis property.

(ii) If $q_{\varphi} < 1$, then the space $CM_{\varphi}[0, \infty)$ is isomorphic to the space $(\bigoplus_{n \in \mathbb N} L_1[0, 1])_{l_\infty}$
 and has the Dunford-Pettis property.
\end{theorem}

Firstly, we prove the following auxiliary result.

\begin{proposition} \label{Pro6}
Let  $w$ be a locally integrable function on $[0, \infty), w(t) > 0$ a.e.,  such that $\int_0^{\infty} w(t)\, dt = \infty$. For $1 \leq p < \infty$ 
we consider the weighted $L_{p}$-space with the norm $\| f\|_{L_p(w)}: = (\int_0^{\infty} |f (t)|^p w(t) \, dt)^{1/p}$. Then the space
$\widetilde{L_p(w)}$ is isomorphic to the space $(\bigoplus_{n \in \mathbb Z} L_{\infty}[0, 1])_{l_p}$ and the constant of 
isomorphism depends only on $w$.
\end{proposition}
\proof
Without loss of generality, we can assume that $\int_0^1 w(t)\, dt = \frac{1}{a - 1}$, where $a > 1$. Thanks to the assumptions, there 
is an increasing sequence $\{t_k\}_{k \in \mathbb Z}$ such that $t_0 = 1, t_n \rightarrow 0$ as $n \rightarrow - \infty, t_n \rightarrow + \infty$ 
as $n \rightarrow + \infty$ and 
\begin{equation} \label{42}
\int_{t_n}^{t_{n+1}} w(t) \, dt = a^n ~ {\rm for ~all} ~ n \in \mathbb Z.
\end{equation}
Then, applying once more \cite[Proposition 2.1]{GHS96} (see also \cite[Lemma 3.2]{GP03}), 
for every $f \in \widetilde{L_p(w)}$ we have
\begin{eqnarray*}
\| f \|_{ \widetilde{L_p(w)}}^p
&=&
\int_0^{\infty}  \widetilde{f(t)}^p w(t)\, dt = \sum_{n \in \mathbb Z} \int_{t_n}^{t_{n+1}} \esssup_{s \geq t} | f(s)|^p w(t) \, dt \\
&\leq&
\sum_{n \in \mathbb Z} \esssup_{s \geq t_n} | f(s)|^p a^n = \sum_{n \in \mathbb Z} \sup_{k \geq n} \esssup_{t_k \leq s \leq t_{k+1}} | f(s)|^p a^n \\
&\leq&
C(a) \, \sum_{n \in \mathbb Z} \| f \chi_{[t_n, t_{n+1}]} \|_{L_\infty}^p a^n = C(a) \, \| f \|_{\oplus}^p,
\end{eqnarray*}
where 
$$
 \| f \|_{\oplus}: = ( \sum_{n \in \mathbb Z} a^n \, \| f \chi_{[t_n, t_{n+1}]} \|_{L_\infty}^p)^{1/p}
$$
and $C(a)$ is some constant depending only on $a$ (and hence on $w$). On the other hand, 
\begin{eqnarray*}
\| f \|_{\oplus}^p 
&=&
a \sum_{n \in \mathbb Z} a^{n-1} \, \| f \chi_{[t_n, t_{n+1}]} \|_{L_\infty}^p 
= a \sum_{n \in \mathbb Z} \int_{t_{n-1}}^{t_n} \esssup_{t_n \leq s \leq t_{n+1}} | f(s)|^p  w(t)\, dt \\
&\leq&
a \sum_{n \in \mathbb Z} \int_{t_{n-1}}^{t_n} \esssup_{s \geq t} | f(s)|^p  w(t)\, dt = a \| f \|_{ \widetilde{L_p(w)}}^p.
\end{eqnarray*}
Since the space $L_{\infty}[a, b]$ is isomorphic to the space $L_{\infty}[0, 1]$ for every $0 < a < b < \infty$, 
the result follows.
\endproof

\proof[Proof of Theorem \ref{Thm5}]
(i) It is clear that $\widetilde{\Lambda_{\varphi}} =\widetilde{L_1(\varphi^{\prime})}$. Therefore, from Proposition \ref{Pro6} it follows that 
$\widetilde{\Lambda_{\varphi}} \, \simeq \,  (\bigoplus_{n \in \mathbb N} L_{\infty}[0, 1])_{l_1}$ and, 
by Bourgain's result \cite[Corollary 7]{Bo81}, 
$\widetilde{\Lambda_{\varphi}}$ has the Dunford-Pettis property.

(ii) Since $q_{\varphi} < 1$, then the operator $C$ is bounded in $M_{\varphi}[0, \infty)$ (cf. \cite[Theorem 6.6, p. 138]{KPS82}) 
and hence, by Theorem A(i), the K\"othe dual of the space $CM_{\varphi}$ coincides with the space $\widetilde{\Lambda_{\psi}}$, 
where $\psi(t) = t/{\varphi(t)}, t > 0$. Therefore, applying Proposition \ref{Pro6}, we are able to get the result arguing in the same 
way as in the concluding part of the proof of Theorem \ref{Thm4a}.
\endproof

\begin{theorem} \label{Thm6}
The spaces $Ces_{\infty}(I)$ and $\widetilde{L_1}(I)$, where $I = [0, \infty)$ or $[0, 1]$, have the Dunford-Pettis property and they 
are not isomorphic.
\end{theorem}
\proof
At first, let $I = [0, \infty)$. Since $ \widetilde{L_1} = \widetilde{\Lambda_{\varphi_1}}$, where $\varphi_1(t) = t$, and 
$Ces_{\infty} = CL_{\infty} = CM_{\varphi_0}$, where $\varphi_0(t) = 1$, by Theorem \ref{Thm5}, the spaces $\widetilde{L_1}$ and 
$Ces_{\infty}$ have the Dunford-Pettis property. 

Let us show that $\widetilde{L_1}$ and $Ces_{\infty}$ are not isomorphic. By Theorem \ref{Thm5}, the space $\widetilde{L_1}$ is 
isomorphic to the space $(\bigoplus_{n \in \mathbb N} L_{\infty}[0, 1])_{l_1}$ and therefore, according to the Pe{\l}czy\'nski result on 
isomorphism between $L_{\infty}$ and $l_{\infty}$ (\cite{Pe58}; see also \cite[Theorem 4.3.10]{AK06}), $\widetilde{L_1}$ is isomorphic 
also to $(\bigoplus_{n \in \mathbb N} l_{\infty})_{l_1}$. Since the latter space fails to contain a complemented subspace isomorphic to 
$L_1[0, 1]$ (see \cite[Proposition 3]{CM08}), then so does not $\widetilde{L_1}$. 
On the other hand, the space $Ces_{\infty}$ contains such a complemented subspace (see Proposition \ref{Pro1n}), and the
result follows.

It is easy to see that the assertion of Proposition \ref{Pro6} holds also for weighted $L_p$-spaces on $[0, 1]$ (with the same proof). 
Therefore, $\widetilde{L_1}[0, 1] \, \simeq \, (\bigoplus_{n \in \mathbb N} L_{\infty}[0, 1])_{l_1}$ and, since 
$(Ces_{\infty}[0, 1])^{\prime} = \widetilde{L_1}[0, 1]$, then  
$Ces_{\infty}[0, 1] \, \simeq \, (\bigoplus_{n \in \mathbb N} L_1[0, 1])_{l_{\infty}}$. 
Thus, the result can be proved in the same way as in the case of $[0,\infty).$
\endproof

Similar result can be deduced from Theorem \ref{Thm5} also for the separable part of a Marcinkie\-wicz space $M_{\varphi}^0$ 
(see also Corollary \ref{Cor5a}). In fact, since the space $C(M_{\varphi}^0)$ on $[0, \infty)$ is order continuous and the condition
$\beta_{M_{\varphi}^0} = q_{\varphi} < 1$ implies the boundedness of the operator $C$ on $M_{\varphi}^0$, by Theorem A(i), we have
$$
[C(M_{\varphi}^0)]^* = [C(M_{\varphi}^0)]^{\prime} = \widetilde{\Lambda_{\psi}}.
$$
As a result, applying Theorem \ref{Thm5}(i), we get that $C(M_{\varphi}^0)$
has the Dunford-Pettis property. However, we prefer to give the more direct proof of the latter fact
(without exploiting Bourgain's results from \cite{Bo81}) by using the following property of separable Ces\'aro-Marcinkiewicz spaces.

\begin{proposition} \label{Prop8}
Suppose that $\varphi$ is an increasing concave function on $[0, \infty)$ such that $\lim_{t \rightarrow 0^+} \varphi(t) = 0$ and $q_{\varphi} < 1$. 
Let $X = C(M_{\varphi}^0)$ on $[0, \infty)$ and let $I_n: = [a_n, b_n]$ be a sequence of intervals from $[0, \infty)$ such that either
\begin{equation} \label{44}
b_1 > a_1 > b_2 > a_2 > \ldots > 0  ~  {\it and} ~ b_n \rightarrow 0^+ ~ {\it as} ~ n \rightarrow \infty
\end{equation}
or 
\begin{equation} \label{45}
a_1 < b_1 < a_2 < b_2 < \ldots   ~  {\it and} ~ a_n \rightarrow \infty ~ {\it as} ~ n \rightarrow \infty. 
\end{equation}
Then there are a subsequence of positive integers $\{n_k\}_{k=1}^{\infty}, n_1 < n_2 < \ldots$ and a constant $C > 0$ such that 
for every sequence $\{x_n\} \subset X$ satisfying the condition: $ \supp x_n \subset I_n, n = 1, 2, \ldots$ we have
\begin{equation} \label{46}
\max_{k = 1, \ldots, m} \| x_{n_k} \|_X \leq \| \sum_{k = 1}^m x_{n_k} \|_X \leq C \max_{k = 1, \ldots, m} \| x_{n_k} \|_X, \, m = 1, 2, \ldots.
\end{equation}
\end{proposition}
\proof
Since a given sequence $\{x_n\}$ under consideration consists of pairwise disjoint functions, the left inequality in (\ref{46}) holds for an arbitrary subsequence 
$\{n_k\}_{k=1}^{\infty}$. So, we need only to prove the reverse inequality. Obviously, we may assume that $x_n \geq 0$ a.e.
Since $q_{\varphi} < 1$, then $\lim_{t \rightarrow 0^+} \frac{t}{\varphi(t)} = 0$. Therefore, in the case (\ref{44}), applying the diagonal 
procedure, from any given sequence $\{I_n\}$ we can extract a subsequence of intervals (which we will denote still by $I_n = [a_n, b_n]$) 
such that
\begin{equation} \label{47}
\sum_{k = n + 1}^{\infty} \psi(b_k) \leq \psi(a_n), ~{\rm where} ~ \psi(t) = t/\varphi(t).
\end{equation}
We claim that the corresponding sequence of functions (still denoting by $\{x_n\}, \supp x_n \subset I_n$) satisfies the right-hand inequality 
in (\ref{46}). For any $m \in \mathbb N$ and $t \in (0, \infty)$ we have
\begin{eqnarray*}
\frac{1}{t} \int_0^t \Big( \sum_{k = 1}^m x_k(s) \Big) \, ds 
&=&
\frac{1}{t} \sum_{j = 2}^m \sum_{i = m - j + 2}^m \int_{a_i}^{b_i} x_i(s) \, ds \, \chi_{[a_{m - j + 1}, b_{m - j + 1}]}(t) \\
& +&
\frac{1}{t} \sum_{j = 1}^m  \int_{a_{m - j + 1}}^{t} x_{m - j + 1}(s) \, ds \, \chi_{[a_{m - j + 1}, b_{m - j + 1}]}(t) \\
& +&
\frac{1}{t} \sum_{j = 1}^m \sum_{i = m - j + 1}^m \int_{a_i}^{b_i} x_i(s) \, ds \, \chi_{[b_{m - j + 1}, a_{m - j}]}(t) \\
&=& 
f_1(t) + f_2(t) + f_3(t),
\end{eqnarray*}
where $a_0 = \infty$ and
$$
f_1(t): = \frac{1}{t} \sum_{j = 2}^m \sum_{i = m - j + 2}^m \int_{a_i}^{b_i} x_i(s) \, ds \, \chi_{[a_{m - j + 1}, a_{m - j}]}(t), 
$$
$$
f_2(t): = \frac{1}{t} \sum_{j = 1}^m  \int_{a_{m - j + 1}}^{b_{m - j + 1}} x_{m - j + 1}(s) \, ds \, \chi_{[b_{m - j + 1}, a_{m - j}]}(t), 
$$
$$
f_3(t): = \frac{1}{t} \sum_{j = 1}^m  \int_{a_{m - j + 1}}^{t} x_{m - j + 1}(s) \, ds \, \chi_{[a_{m - j + 1}, b_{m - j + 1}]}(t).
$$
Since $\beta_{M_{\varphi}^0} = q_{\varphi} < 1$, the operator $C$ is bounded in $M_{\varphi}^0$ (cf. \cite[Theorem 6.6, 
p. 138]{KPS82}) and hence, by Theorem A(i), $(CX)^* = (CX)^{\prime} = \widetilde{X^{\prime}} = \widetilde{\Lambda_{\psi}}$ with 
equivalent norms. Thus, by ({\ref{HRI}) and (\ref{equality}), for arbitrary $0 < a < b < \infty$ 
and $x \in X$
\begin{equation} \label{48}
\int_a^b |x (s)| \, ds \leq C_1 \, \| x \|_X \| \chi_{[a, b]} \|_{\widetilde{\Lambda_{\psi}}} = C_1 \, \| x \|_X \psi(b).
\end{equation}
Hence, by (\ref{47}),
\begin{eqnarray*}
f_1(t)
&\leq&
\frac{C_1}{t} \sum_{j = 2}^m \sum_{i = m - j + 2}^m \psi(b_i)\, \| x_i \|_X \, \chi_{[a_{m - j + 1}, a_{m - j}]}(t) \\
&\leq&
\frac{C_1}{t} \sum_{j = 2}^m  \psi(a_{m - j + 1})\, \chi_{[a_{m - j + 1}, a_{m - j}]}(t) \cdot  \max_{i = 1, \ldots, m} \| x_i \|_X \\
&\leq&
\frac{C_1}{\varphi(t)} \sum_{j = 2}^m \chi_{[a_{m - j + 1}, a_{m - j}]}(t) \cdot  \max_{i = 1, \ldots, m} \| x_i \|_X \leq
\frac{C_1}{\varphi(t)}  \max_{i = 1, \ldots, m} \| x_i \|_X, 
\end{eqnarray*}
whence since $q_{\varphi} < 1$, by (\ref{estimate2.13}), it follows that
\begin{equation} \label{49}
\| f_1 \|_{M_{\varphi}} \leq C_1 \frac{\varphi(t)}{t} \int_0^t \frac{1}{\varphi(s)} ds \max_{i = 1, \ldots, m} \| x_i \|_X \leq 
C_1 C_2 \max_{i = 1, \ldots, m} \| x_i \|_X.
\end{equation}
Similarly, since $\psi$ increases, we have
\begin{eqnarray*}
f_2(t)
&\leq&
\frac{C_1}{t} \sum_{j = 1}^m \psi(b_{m - j + 1}) \| x_{m - j + 1}\|_X \cdot \chi_{[b_{m - j + 1}, a_{m - j}]}(t) \\
&\leq&
\frac{C_1}{\varphi(t)} \sum_{j = 1}^m \chi_{[b_{m - j + 1}, a_{m - j}]}(t) \max_{i = 1, \ldots, m} \| x_i \|_X \leq 
\frac{C_1}{\varphi(t)} \max_{i = 1, \ldots, m} \| x_i \|_X,
\end{eqnarray*}
and again
\begin{equation} \label{4.10}
\| f_2 \|_{M_{\varphi}} \leq C_1 C_2 \max_{i = 1, \ldots, m} \| x_i \|_X.
\end{equation}
At final, using (\ref{48}) once more, we have
\begin{eqnarray*}
f_3(t)
&\leq&
\frac{C_1}{t} \sum_{j = 1}^m \psi(t) \| x_{m - j + 1} \|_X \cdot \chi_{[a_{m - j + 1}, b_{m - j +1}]}(t) \\
&\leq&
\frac{C_1}{\varphi(t)} \sum_{j = 1}^m \chi_{[a_{m - j + 1}, b_{m - j +1}]}(t) \max_{i = 1, \ldots, m} \| x_i \|_X 
\leq \frac{C_1}{\varphi(t)}  \max_{i = 1, \ldots, m} \| x_i \|_X,
\end{eqnarray*}
whence again
\begin{equation} \label{4.11}
\| f_3 \|_{M_{\varphi}} \leq C_1 C_2 \max_{i = 1, \ldots, m} \| x_i \|_X.
\end{equation}
Thus, from (\ref{48})--(\ref{4.11}) it follows that
$$
\| \sum_{k = 1}^m x_k \|_X \leq \| f_1 \|_{M_{\varphi}} + \| f_2 \|_{M_{\varphi}} + \| f_3 \|_{M_{\varphi}} \leq C \max_{i = 1, \ldots, m} \| x_i \|_X,
$$
where the constant $C: = 3 C_1C_2$ depends only on the function $\varphi$.

Regarding the case (\ref{45}), we note that condition $q_{\varphi} < 1$ implies $\lim_{t \rightarrow \infty} \frac{t}{\varphi(t)} = \infty$. 
Hence, from any given sequence of intervals we can select a subsequence of intervals (denoting still by $I_n = [a_n, b_n]$) such that
\begin{equation} \label{4.12}
\sum_{i = 1}^{k-1} \psi(b_i) \leq \psi (a_k), k = 2, 3, \ldots .
\end{equation}
For arbitrary $m \in \mathbb N$ and $t > 0$ we have
\begin{eqnarray*}
\frac{1}{t} \int_0^t \Big( \sum_{k = 1}^m x_k(s) \Big) \, ds 
&=&
\frac{1}{t} \sum_{k = 1}^m \Big (\sum_{i = 1}^{k - 1} \int_{a_i}^{b_i} x_i(s) \, ds  + \int_{a_k}^t x_k(s) \, ds \Big) \chi_{[a_k, b_k]}(t) \\
& +&
\frac{1}{t} \sum_{k = 1}^m \sum_{i = 1}^k \int_{a_i}^{b_i} x_i(s) \, ds \cdot \chi_{[b_k, a_{k + 1}]}(t) = g_1(t) + g_2(t) + g_3(t),
\end{eqnarray*}
where $a_{m+1} = \infty$ and
$$
g_1(t): = \frac{1}{t} \sum_{k = 2}^m \sum_{i = 1}^{k - 1} \int_{a_i}^{b_i} x_i(s) \, ds \cdot \chi_{[a_k, a_{k + 1}]}(t), 
$$
$$
g_2(t): = \frac{1}{t} \sum_{k = 1}^m  \int_{a_k}^{b_k} x_i(s) \, ds \cdot \chi_{[b_k, a_{k+1}]}(t), 
$$
$$
g_3(t): = \sum_{k = 1}^m  \int_{a_k}^t x_k(s) \, ds \cdot \chi_{[a_k, b_k]}(t).
$$
Firstly, applying (\ref{48}) and (\ref{4.12}) we obtain
\begin{eqnarray*}
g_1(t) 
&\leq&
\frac{C_1}{t} \sum_{k = 2}^m \sum_{i = 1}^{k - 1} \psi(b_i) \| x_i\|_X \, \cdot \chi_{[a_k, a_{k + 1}]}(t) \\
&\leq&
\frac{C_1}{t} \sum_{k = 2}^m \psi(a_k) \cdot \chi_{[a_k, a_{k + 1}]}(t) \cdot  \max_{i = 1, \ldots, m} \| x_i \|_X \\
&\leq&
\frac{C_1}{\varphi(t)} \sum_{k = 2}^m \chi_{[a_k, a_{k + 1}]}(t) \cdot  \max_{i = 1, \ldots, m} \| x_i \|_X 
\leq \frac{C_1}{\varphi(t)} \max_{i = 1, \ldots, m} \| x_i \|_X.
\end{eqnarray*}
Next, similarly,
\begin{eqnarray*}
g_2(t) 
&\leq&
\frac{C_1}{t} \sum_{k = 1}^m  \psi(b_k) \| x_k \|_X \cdot \chi_{[b_k, a_{k+1}]}(t)
\leq \frac{C_1}{\varphi(t)} \max_{i = 1, \ldots, m} \| x_i \|_X
\end{eqnarray*}
and
\begin{equation*}
g_3(t) \leq \frac{C_1}{t} \sum_{k = 1}^m  \psi(t) \| x_k \|_X \cdot \chi_{[a_k, b_k]}(t) \leq 
\frac{C_1}{\varphi(t)}  \max_{i = 1, \ldots, m} \| x_i \|_X.
\end{equation*}
As a result, using (\ref{estimate2.13}), we have
$$
\| \sum_{k = 1}^m x_k \|_X \leq \| g_1 \|_{M_{\varphi}} + \| g_2 \|_{M_{\varphi}} + \| g_3 \|_{M_{\varphi}} \leq C \max_{i = 1, \ldots, m} \| x_i \|_X,
$$
where the constant $C: = 3 C_1C_2$ depends only on the function $\varphi$.
\endproof

\begin{corollary} \label{Cor9}
Let $\varphi$ satisfy all the conditions of Proposition \ref{Prop8} and let $X = C(M_{\varphi}^0)$ on $[0, \infty)$. Suppose
that $I_n: = [a_n, b_n], n = 1, 2, \ldots$, be a sequence of intervals from $[0, \infty)$ such that either
$b_1 > a_1 > b_2 > a_2 > \ldots > 0$ and $b_n \rightarrow 0^+$ as $n \rightarrow \infty$ or 
$a_1 < b_1 < a_2 < b_2 < \ldots$ and $a_n \rightarrow \infty$ as $n \rightarrow \infty$. Then, every semi-normalized sequence
$\{f_n\} \subset X$ such that $\supp f_n \subset I_n, n = 1, 2, \ldots$ contains a subsequence $\{f_{n_k}\}$ which is equivalent 
in $X$ to the canonical basis in $c_0$.
\end{corollary}
\proof
At first, applying Proposition \ref{Prop8}, we find a subsequence of positive integers $\{n_k\}_{k = 1}^{\infty}, \newline 
n_1 < n_2 < \ldots$ and a constant $C > 0$ such that for every sequence $\{x_n\} \subset X$ with $ \supp x_n \subset I_n, n = 1, 2, \ldots,$ 
we have
$$
\max_{k = 1, \ldots, m} \| x_{n_k} \|_X \leq \| \sum_{k = 1}^m x_{n_k} \|_X \leq C \, \max_{k = 1, \ldots, m} \| x_{n_k} \|_X, ~ m = 1, 2, \ldots .
$$
In particular, setting $x_n = c_k f_n$ if $n_k \leq n < n_{k + 1}, k = 1, 2, \ldots$, where $(c_k)$ is an arbitrary sequence from $c_0$, and 
assuming that $D^{-1} \leq \| f_n \|_X \leq D,$ $n = 1, 2, \ldots$ for all $m \in \mathbb N,$ we obtain
$$
D^{-1} \max_{k = 1, \ldots, m} | c_k | \leq \| \sum_{k = 1}^m c_k f_{n_k} \|_X \leq C \,D \max_{k = 1, \ldots, m} | c_k |.
$$
Since $(c_k) \in c_0$, then the series $\sum_{k = 1}^{\infty} c_k f_{n_k}$ converges in $X$ and we have
$$
D^{-1} \| (c_k) \|_{c_0} \leq \| \sum_{k = 1}^{\infty} c_k f_{n_k} \|_X \leq C \,D \| (c_k) \|_{c_0}.
$$ 
\endproof

\begin{theorem}\label{Thm7n}
Let $\varphi$ be an increasing concave function on $[0, \infty)$ such that $\lim_{t \rightarrow 0^+} \varphi(t) = 0$ and $q_{\varphi} < 1$. 
Then, the space $X = C(M_{\varphi}^0)$ on $[0, \infty)$ has the Dunford-Pettis property. 
\end{theorem}
\proof
On the contrary, assuming that $X$ does not have the Dunford-Pettis property, we can find sequences $\{u_n\} \subset X$ such that 
$\| u_n \|_X = 1, u_n \rightarrow 0$ weakly in $X$ and $\{v_n\} \subset X^* = X^{\prime} = \widetilde{\Lambda_{\psi}}$ such that 
$\| v_n \|_{X^{\prime}} = 1, v_n \rightarrow 0$ weakly in $X^{\prime}$ satisfying the condition
\begin{equation} \label{4.13}
\langle u_n, v_n \rangle : = \int_0^{\infty} u_n(t) v_n(t) \, dt \geq \delta,
\end{equation}
for some $\delta > 0$ and all $n \in \mathbb N$. It is easy to see that $u_n \chi_{[a, b]} \rightarrow 0$ weakly in $X$ for every 
$0 < a < b < \infty$. In fact, if $v \in X^{\prime}$, then $\langle u_n \chi_{[a, b]}, v \rangle = \langle u_n, v \chi_{[a, b]} \rangle \rightarrow 0$ 
as $n \rightarrow \infty$, because of $ v \chi_{[a, b]} \in X^{\prime}$. Moreover (see Proposition \ref{Pro1n}),  
$$
X_{|_{ [a, b]}}: = \{u \in X: \supp u \subset [a, b] \} = L_1[a, b]
$$
with equivalence of norms, and therefore $u_n \chi_{[a, b]} \rightarrow 0$ weakly in $L_1[a, b]$. Setting 
$\alpha_n(u): = \int_a^b u(t) v_n(t) \, dt, n \in \mathbb N$, we see that $\alpha_n \in (L_1[a, b])^* = L_{\infty}[a, b] = X^{\prime}/M$, 
where $M = \{v \in X^{\prime}: \langle u, v \rangle  = 0$ for all $u \in L_1[a, b]$\}. Then,
$$
(L_{\infty}[a, b])^* = (X^{\prime}/M)^* = \{ F \in (X^{\prime})^*: F(v) = 0 ~ {\rm for ~ all} ~ v \in M\},
$$
and therefore $(L_{\infty}[a, b])^* \subset (X^{\prime})^*$. Hence, from the fact that $ v_n \rightarrow 0$ weakly in $X^{\prime}$ it follows 
that $\alpha_n \rightarrow 0$ weakly in $L_{\infty}[a, b]$. Since $L_1[a, b]$ has the Dunford-Pettis property, as a result we have
$$
\alpha_n (u_n \cdot \chi_{[a, b]}) = \int_a^b u_n(t) v_n(t)\, dt \rightarrow 0 ~ {\rm as} ~ n \rightarrow \infty,
$$
for every $0<a<b<\infty.$ Thus, taking into account (\ref{4.13}), we can select subsequences of $\{u_n\}$ 
and $\{v_n\}$ (we will denote them still by $\{u_n\}$ and 
$\{v_n\}$) such that at least one of the following conditions holds:

(a) there exists a sequence $\{b_n\}_{n = 1}^{\infty}$ with $b_1 > b_2 > \ldots, \lim_{n \rightarrow \infty} b_n = 0$ and
$$
\int_0^{b_n} u_n(t) v_n(t) \, dt \geq \frac{3 \delta}{4}, ~ n \in \mathbb N;
$$

(b) there exists a sequence $\{a_n\}_{n = 1}^{\infty}$ with $a_1 < a_2 < \ldots, \lim_{n \rightarrow \infty} a_n = \infty$ and
$$
\int_{a_n}^{\infty} u_n(t) v_n(t) \, dt \geq \frac{3 \delta}{4}, ~ n \in \mathbb N.
$$
Since $\int_0^{\infty} |u_n(t) v_n(t)| \, dt < \infty$ for every $n \in \mathbb N$, passing to further subsequences, we can find a sequence 
of intervals $I_n = [a_n , b _n], n = 1, 2, \ldots$ such that either $b_1 > a_1 > b_2 > a_2 > \ldots, \lim_{n \rightarrow \infty} b_n = 0$, or 
$a_1 < b_1 < a_2 < b_2 < \ldots, \lim_{n \rightarrow \infty} a_n = \infty$, for which
\begin{equation} \label{4.14}
\int_{I_n}u_n(t) v_n(t) \, dt \geq \frac{\delta}{2}, ~n \in \mathbb N.
\end{equation}
Now, we set $f_n: = u_n \cdot \chi_{I_n}, n = 1, 2, \ldots$. From (\ref{4.14}) it follows that $\{ f_n\}$ is a semi-normalized sequence 
in $X$. So, applying Corollary \ref{Cor9}, we can extract a subsequence (denoted still by $\{f_n\}$), which is equivalent in $X$ to the 
canonical basis in $c_0$. Therefore, $f_n \rightarrow 0$ weakly in the closed linear span $[f_n]$ (and in $X$). Clearly, 
$\theta_n(f): = \int_0^{\infty} f(t) v_n(t) \, dt$ is a bounded linear functional on $[f_n]$. As above, $[f_n]^{**} \subset (X^{\prime})^*$. 
Therefore, since $v_n \rightarrow 0$ weakly in $X^{\prime}$, we have $\theta_n \rightarrow 0$ weakly in $[f_n]^*$. Noting that the 
subspace $[f_n]$ is isomorphic to $c_0$, which has the Dunford-Pettis property, we obtain
$$
\int_{I_n}u_n(t) v_n(t) \, dt = \theta_n(f_n) \rightarrow 0 ~ {\rm as} ~ n \rightarrow \infty,
$$
which contradicts (\ref{4.14}). Thus, the proof is completed.
\endproof

As we know, the condition $0 < p_{\varphi} \leq q_{\varphi} < 1$ guarantees that $C\Lambda_{\varphi}$ on 
$[0, \infty)$ is a weighted $L_1$-space up to equivalence of norms (see \cite[Theorem 4.4]{DS07} and \cite{LM15a}). 
It turns out that similar result holds also for the Ces\`aro-Lorentz spaces on $[0, 1]$.

\begin{theorem} \label{Thm8new}
Let $\varphi$ be an increasing concave function on $[0,1]$ such that $0 < p_{\varphi}^0 \leq q_{\varphi}^0 < 1$. Then 
$$
C\Lambda_{\varphi}[0,1] = L_1(w), ~ {\it with} ~ w(t) = \int_0^{1-t}\frac{\varphi'(s)}{t+s} \, ds.
$$
\end{theorem}
\proof By duality and Fubini's theorem, we have 
\begin{eqnarray*}
\|f\|_{C\Lambda_{\varphi}} 
&=&
\sup_{\| g\|_{\Lambda_{\varphi}' \leq 1}} \int_0^1C|f|(x) |g(x)| \, dx =
\sup_{\| g \|_{\Lambda_{\varphi}' \leq 1}} \int_0^1|g(x)| \Big(\frac{1}{x} \int_0^x |f(t)|\, dt \Big) dx \\
&=&
\sup_{\| g \|_{\Lambda_{\varphi}' \leq 1}} \int_0^1|f(t)| \Big(\int_t^1\frac{|g(x)|}{x} dx \Big) dt \leq
\int_0^1|f(t)| \, \| h_t \|_{\Lambda_{\varphi}} dt
\end{eqnarray*}
where $h_t(x) = \frac{1}{x} \chi_{[t, 1]}(x)$. Then
$$
\|h_t\|_{\Lambda_{\varphi}} = \int_0^1(h_t)^*(s)\varphi^{\prime}(s) \, ds = \int_0^{1-t} \frac{\varphi^{\prime}(s)}{s+t} \,ds = w(t)
$$
and consequently the above inequality means that $L_1(w) \overset{1}{\hookrightarrow } C\Lambda_{\varphi}$. 
In view of the conditions imposed on the indices $p_{\varphi}^0$ and $q_{\varphi}^0$ the operators $C$ and $C^*$
are bounded in $\Lambda_{\varphi}$ (see \cite[Chapter~II, \S\,8.6]{KPS82}). Therefore, the reverse inclusion
is equivalent, by duality (see Theorem A(ii)), to the following one
$$
L_{\infty}(1/w) \, {\hookrightarrow } \, \widetilde{M_{\psi}(v)},
$$
where $\psi(t)=\frac{t}{\varphi(t)}$ and $v(t)=\frac{1}{1-t}$. Thus, it is enough to check that $w \in \widetilde{M_{\psi}(v)}$, 
i.e.,
$$
\| w \|_{\widetilde{M_{\psi}(v)}} = \sup_{0 < t \leq 1}\frac{1}{\varphi(t)}\int_0^t (v\widetilde{w})^*(x) \, dx<\infty.
$$
First of all notice that $w$ is decreasing, so we have $\widetilde{w} = w$. We divide the function $v\cdot w$ into two parts, namely,
$$
v(t) w(t)=\frac{w(t)}{1-t}\chi_{[0,1/2]}(t)+\frac{w(t)}{1-t}\chi_{[1/2,1]}(t) = w_0(t)+w_1(t).
$$
Thus, we need only to check that $w_0$ and $w_1$ belong to the space $M_{\psi}$. 

By Fubini's theorem, we have
\begin{eqnarray*}
\int_0^x w^*(t) \, dt 
&=&
\int_0^x \Big(\int_0^{1-t}\frac{\varphi'(s)}{t+s} \,dt \Big) ds \\
&=& 
\int_0^{1-x} \Big(\int_0^{x}\frac{\varphi'(s)}{t+s}dt \Big) ds + \int_{1-x}^{1}\Big(\int_0^{1-s}\frac{\varphi'(s)}{t+s}dt\Big)ds\\
&=&
\int_0^{1-x}\varphi'(s)\ln\frac{x+s}{s}ds + \int_{1-x}^1\varphi'(s)\ln\frac{1}{s}ds.
\end{eqnarray*}
Then, since $0< x <1/2$, the second summand can be estimated thanks to monotonicity and subadditivity of the concave function 
$\varphi$ as follows 
$$
\int_{1-x}^1\varphi'(s)\ln\frac{1}{s}ds  \leq \ln 2 \int_{1-x}^1\varphi'(s) ds = \ln 2 \, [\varphi(1)-\varphi(1-x)] \leq \ln2\, \varphi(x).
$$
While for the first one, integrating by parts, we get
$$
\int_0^{1-x}\varphi'(s)\ln\frac{x+s}{s} ds = 
\varphi(1-x) \ln \frac{1}{1-x} - \lim_{s \rightarrow 0^+} \varphi(s) \ln \frac{x+s}{s} 
$$
$$
+ x \int_0^{1-x}\frac{\varphi(s)}{(s+x) s} \,ds \leq \varphi (1) \frac{x}{1-x} + x \int_0^{1-x} \frac{\varphi(s)}{(s+x) s} \, ds.
$$
Since $0 < x< 1/2$, then by concavity of $\varphi$, we get
$$
\varphi (1) \frac{ x}{1-x} \leq 2\, \varphi (1) x \leq 2\, \varphi(x).
$$
Moreover, for some $0< a < 1, A \geq 1$ and all $0< x<1, t > 0$ we have $\varphi(tx) \leq A t^{a} \varphi(x)$ and consequently, 
putting $s = t x$, we obtain
\begin{eqnarray*}
\int_0^{1-x}\frac{\varphi(s)}{(s+x)s} ds
&=&
\frac{1}{x} \int_0^{\frac{1-x}{x}}\frac{\varphi(t x)}{(1+t) t} dt
\leq \frac{A}{x} \int_0^{\frac{1-x}{x}}\frac{t^{a} \varphi(x)}{(1+t) t}dt \\
&\leq& 
A  \frac{\varphi(x)}{x} \int_0^{\infty}\frac{t^{a-1}}{1+t} dt = B \, \frac{\varphi(x)}{x}.
\end{eqnarray*}
Thus, for $0 < x< 1/2$,
\begin{eqnarray*}
\int_0^x (w_0)^*(t) dt 
&\leq&
2 \int_0^{x} (w \chi_{(0, 1/2]})^*(t) dt \leq 2 \int_0^{x} w^*(t) dt \\
&\leq&
2 (2 + B + \ln 2) \, \varphi(x), 
\end{eqnarray*}
whence $w_0\in M_{\psi}$. 

Let us consider now $w_1$. For $1/2 < t \leq 1$ we have
$$
w_1(t) = \frac{1}{1-t}\int_0^{1-t}\frac{\varphi'(s)}{t+s} \, ds \leq \frac{2}{1-t}\int_0^{1-t}\varphi'(s) \, ds = 2 \, \frac{\varphi(1-t)}{1-t}.
$$
Since the function $\frac{\varphi(1-t)}{1-t}$ is increasing, we conclude that $w_1^*(t) \leq 2 \, \frac{\varphi(t)}{t},$ $0<t\leq 1$. 
In consequence, from ({\ref{estimate2.13}) and the condition $q_{\psi}^0 = 1 - p_{\varphi}^0 < 1$ it follows that
$$
\int_0^{x} w_1^*(t) dt \leq 2 \int_0^{x} \frac{\varphi(t)}{t} dt = 2 \, \int_0^{x} \frac{1}{\psi(t)} dt \leq 2 \, C  \frac{x}{\psi(x)} 
= 2 \, C \varphi(x),
$$
which finishes the proof.
\endproof

Of course, from Theorem \ref{Thm8new} it follows that the space $C\Lambda_{\varphi}[0,1]$ has 
the Dunford-Pettis property whenever $0 < p_{\varphi}^0 \leq q_{\varphi}^0 < 1$. 
Let us prove analogous result for separable Ces\`aro-Marcinkiewicz spaces.

\begin{theorem} \label{Thm9new}
Let $\varphi$ be an increasing concave function on $[0,1]$ such that $\lim_{t \rightarrow 0^+} \varphi(t) = 0$ and $q_{\varphi}^0 < 1$. 
Then the space $C(M_{\varphi}^0)[0,1]$ has the Dunford-Pettis property.
\end{theorem}
\proof For every $k = 2, 3, \ldots$ we set
$$
X_k: = {CM_{\varphi}^0}_{ \big |_{[0, 1-1/k]}} = \{f \in CM_{\varphi}^0: \supp f \subset [0, 1 - \frac{1}{k}] \}.
$$
Since $CM_{\varphi}^0$ is an order continuous space, the union $\bigcup_{k=2}^{\infty} X_k$ is dense in it. Moreover, from the 
definition of Ces\`aro spaces it follows that, for every $k = 2, 3, \ldots$, $X_k$ can be regarded as a complemented subspace of 
the space $CM_{\varphi_1}^0[0,\infty)$, where $\varphi_1$ is a concave extension of the function $\varphi$ to the semi-axis $[0, \infty)$ 
such that $q_{\varphi_1} < 1$. (Notice that $CX[0,1]$ is not a restriction of $CX[0,\infty)$ to the interval $[0,1]$. More precisely, 
similarly as in \cite[Remark 5]{AM09}, one can check that $CX[0, \infty) {\large _{|[0, 1]}} = CX[0, 1] \cap L_1[0,1]$).
Therefore, an inspection of the proof of Theorem \ref{Thm5} shows that 
$X_k \simeq (\bigoplus_{n \in \mathbb N} L_1[0, 1])_{l_\infty}$, whence the space $(\bigoplus_{k =2}^{\infty} X_k)_{l_\infty}$ is isomorphic 
to the latter $l_{\infty}$-sum as well. Thus, $(\bigoplus_{k =2}^{\infty} X_k)_{l_\infty}$ has the Dunford-Pettis property.
Finally, applying Proposition 2 from \cite{Bo81}, we conclude that $CM_{\varphi}^0$ also posseses the latter property, and the proof 
is completed.
\endproof

\begin{remark} \label{Rem2} 
The assertion of Theorem \ref{Thm9new} cannot be deduced from Theorem \ref{Thm5}(i), using the above Bourgain's results,  
because of the difference in the duality results for Ces\`aro spaces for the cases of $[0,1]$ and $[0,\infty)$ (see Theorem A). 
We would like to mention here also that we couldn't identify conditions under which the space $CM_{\varphi}[0, 1]$ has the 
Dunford-Pettis property.
\end{remark}

Now, we present some negative results related to the Dunford-Pettis property of Ces\`aro and Tandori function spaces. 

\begin{theorem}\label{Thm7new}
Let $X$ be a reflexive symmetric function space on $I$ such that the operator $C$ is bounded on $X$.
 
(i) If $I = [0, \infty)$, then the spaces $CX$ and $\widetilde{X^{\prime}}$ do not have the Dunford-Pettis property. 

(ii) If $I = [0,1], X$ has the Fatou property and the operator $C^*$ is bounded on $X$, then the spaces
$CX$ and $ \widetilde{X^{\prime}}$ do not have the Dunford-Pettis property. 
\end{theorem}
\proof
(i) The proof is rather similar to the proof in the sequence case (Theorem \ref{Thm3}). Again it is sufficient to prove that the operator 
$C: CX \rightarrow X$ is not a Dunford-Pettis operator. Let us show that 
$x_n = \dfrac{1}{\varphi_X(1/n)} \, \chi_{[0, 1/n]}, \,n = 1, 2, \dots,$ is a weakly null sequence in $CX.$
Since $X$ is order continuous, it follows that $CX$ is also order continuous and by Theorem A(i) we obtain
$(CX)^*= (CX)' = \widetilde{X'}$. Thus, we need only to check that
\begin{equation}\label{43}
\langle y, x_{n}\rangle = \frac{1}{\varphi_{X}(1/n)}\int_0^{1/n} y(t)\,dt \to 0\;\;\mbox{as}\;\;n\to\infty,
\end{equation}
for every decreasing positive function $y \in X'$. Again $X'\subset M_{\varphi_{X'}},$
where $M_{\varphi_{X'}}$ is the Marcinkiewicz function space with the fundamental function $\varphi_{X'}$.
By reflexivity of $X$ we have $X' = (X')^0,$ and thus  
$$
X'\subset M_{\varphi_{X'}}^0 \subset\Big\{z = z(t): \lim_{t\rightarrow 0}\frac{\varphi_{X'}(t)}{t}\int_0^t z^*(s)\,ds = 0\Big\}.
$$
But $\varphi_X(t)=t/\varphi_{X'}(t)$ and \eqref{43} follows from the above embedding.
On the other hand, $Cx_n\ge x_n$ and so $\|Cx_n\|_X \ge\|x_n\|_X = 1$. This means that $CX$ does not have the Dunford-Pettis 
property. Moreover, since $\widetilde{X'} = (CX)^{\prime} = (CX)^*$, then $\widetilde{X'}$ fails to have the latter property as well.

(ii) The only difference of this case from the case of $[0,\infty)$ is the fact that now $(CX)^* = (CX)' = \widetilde{X'(1/(1-t))}.$ However, ``near zero'' the latter space coincides with
the space $\widetilde{X'}$ without a weight. Thus, we can repeat the same proof as in (i).
\endproof

As we know (see \cite[Theorem 4.4]{DS07} and Theorem~\ref{Thm8new}), the Ces\`aro-Lorentz spaces may coincide with weighted 
$L_1$-spaces and therefore may be isomorphic to the symmetric space $L_1.$ At the same time, it is not the case for Ces\`aro spaces $CX$ 
when $X$ is reflexive.

\begin{corollary} \label{Cor2n}
If $X$ is a reflexive symmetric function space on $[0, 1]$ such that the operator $C$ is bounded on $X$, then $CX$ is not 
isomorphic to any symmetric space on $[0,1]$. 
\end{corollary}
\proof
Suppose $CX$ is isomorphic to some symmetric space $Y$ on $[0, 1]$. Hence, by Proposition \ref{Pro1n}, $Y$ contains a complemented 
copy of $L_1[0,1]$. On the other hand, as Kalton proved in \cite[Theorem 7.4]{Ka93}, every separable symmetric 
space on $[0,1]$ that contains a complemented subspace isomorphic to $L_1[0,1]$ is isomorphic to $L_1[0,1]$ itself. Therefore, 
we conclude that $Y$ is isomorphic to $L_1[0,1]$. On the other hand, $CX \simeq Y$ cannot be isomorphic to $L_1[0,1]$, because by  
Theorem \ref{Thm7new} that space fails to have the Dunford-Pettis property.

\section{Isomorphism between $Ces_{\infty}$ and $ces_{\infty}$}

In \cite[Theorem 9]{AM09} (see also \cite[Theorem 7.2]{AM14}) it was proved that the spaces $Ces_{\infty}[0, 1]$ and 
$Ces_{\infty}[0, \infty)$ are isomorphic and there the question was raised if the spaces $Ces_{\infty}$ and $ces_{\infty}$ are isomorphic 
(cf. \cite[Problem 1]{AM09} and \cite[Problem 4]{AM14}). The following theorem solves this problem in affirmative.

\begin{theorem}\label{mainthm}
The spaces $Ces_{\infty}$ and $ces_{\infty}$ are isomorphic.  
\end{theorem}

\proof At first, we recall that, by Corollary \ref{Cor3new}, $ces_{\infty} \simeq (\bigoplus_{n=0}^{\infty} l_1^{2^n})_{l_\infty}$ and, 
by Theorem \ref{Thm5}, $Ces_{\infty} \simeq (\bigoplus_{n \in \mathbb N} L_1[0, 1])_{l_\infty}$. Therefore, 
$ces_{\infty} \simeq ces_{\infty} \oplus ces_{\infty}$ and $Ces_{\infty} \simeq  Ces_{\infty}\oplus Ces_{\infty}$, which 
shows that we can apply Pe{\l}czy\'nski decomposition argument (see \cite[Proposition 4]{Pe60} or \cite[Theorem 2.2.3]{AK06}).
In other words, the proof will be completed whenever we check that $ces_{\infty}$ is isomorphic to a complemented subspace of
$Ces_{\infty}$ and vice versa.

Clearly, for every $n = 0, 1, 2, \ldots$ the space $l_1^{2^n}$ can be complementably embedded into the space $L_1[0, 1]$. Therefore, 
the fact that $ces_{\infty}$ is isomorphic to a complemented subspace of $Ces_{\infty}$ follows at once from the above isomorphic 
representations of these spaces.

Just a little more efforts are required for the proof of the reverse statement. Let us represent the set $\mathbb{N}\cup \{0\}$ as 
a union of infinite increasing pairwise disjoint sequences $(a_n^k)_{n=0}^{\infty}$, $k=1,2,\dots$ Then, we can write 
\begin{equation} \label{eqThm11}
ces_{\infty}\simeq \Big(\bigoplus_{k=1}^{\infty}(\bigoplus_{n=0}^{\infty} l_1^{2^{a_n^k}})_{l_\infty}\Big)_{l_\infty}.
\end{equation}
Since $n\leq a_n^k$, where $k = 1, 2, \dots$ and $n = 0, 1, 2, \ldots$ are arbitrary, the space $l_1^{2^{n}}$ can be considered as 
a complemented subspace of the space $l_1^{2^{a_n^k}}$. Let $P_n^k$ be a respective projection and $P_k = \bigoplus_{n=0}^{\infty} P_n^k$.
Noting that $P_k \Big((\bigoplus_{n=0}^{\infty} l_1^{2^{a_n^k}})_{l_\infty} \Big) = (\bigoplus_{n=0}^{\infty} l_1^{2^n})_{l_{\infty}}$, 
we see that, by Theorem \ref{Thm3}, $L_1[0, 1]$ is complemented in 
$P_k \Big((\bigoplus_{n=0}^{\infty} l_1^{2^{a_n^k}})_{l_\infty} \Big)$ and hence in the space $(\bigoplus_{n=0}^{\infty} l_1^{2^{a_n^k}})_{l_\infty}$. 
At final, from (\ref{eqThm11}) it follows that $Ces_{\infty} \simeq (\bigoplus_{n \in \mathbb N} L_1[0, 1])_{l_\infty}$
is isomorphic to a complemented subspace of $ces_{\infty}$ and the proof is completed.
\endproof

\begin{corollary}\label{Cor7n}
The space $Ces_{\infty}(I)$, where $I = [0, \infty)$ or $[0, 1]$, is isomorphic to a dual space. 
\end{corollary}
\proof By (\ref{Alex1}) and Theorem \ref{mainthm} we have
$(\widetilde{l_1})^* = (\widetilde{l_1})^{\prime} = ces_{\infty} \simeq Ces_{\infty}$.
\endproof

In contrast to the latter result, order continuous Ces\`aro spaces fail to be isomorphic to the dual ones.

\begin{proposition} \label{Cor11n}
If $X$ is a symmetric function space on $I = [0, 1]$ or $I = [0, \infty)$ such that $X$ is order continuous and $C$ 
is bounded on $X$, then $CX$ is not isomorphic to a dual space. 
\end{proposition}
\proof
Suppose that $CX$ is isomorphic to a dual space. By Proposition \ref{Pro1n}, $CX$ contains a complemented 
subspace isomorphic to $L_1[0,1]$. Therefore, applying Hagler-Stegall theorem (cf. \cite[Theorem 1]{HS73}), we 
see that $CX$ contains also a subspace isomorphic to $C[0, 1]^*$. However, it is impossible, since $CX$ is 
separable (by Lemma 1 in \cite{LM15b}).
\endproof

Let us comment the latter results. Suppose that $X$ is an ideal Banach function space with the Fatou property such that 
the separable part of its K\"othe dual $(X^{\prime})^0$ has the same support as $X$ itself. Then, an easy argument 
shows that
$$
[(X^{\prime})^0]^* = [(X^{\prime})^0]^{\prime} = X^{\prime \prime} = X, 
$$
i.e., $X$ is a dual space. So, the space $(X^{\prime})^0$ is a natural candidate for being predual of a dual ideal Banach 
space $X$. Moreover, as we have seen, separable $CX$ spaces are not isomorphic to dual ones similarly as $L_1$ 
and both of them have K\"othe dual without nontrivial absolutely continuous elements. Hence, the following conjecture 
may arise: an ideal Banach space whose K\"othe dual has trivial subspace of order continuous elements is not isomorphic 
to a dual space. This statement, however, is false; by Corollary \ref{Cor7n}, the Ces\`aro space $Ces_{\infty}$, satisfying 
$[(Ces_{\infty})^{\prime}]^0 = (\widetilde{L_1})^0=\{0\}$, is a dual space. In connection with that we can ask, for example, 
if the symmetric space $X = L_1 + L_{\infty}$ on $[0, \infty)$ is isomorphic to a dual space noting that 
$(X^{\prime})^0 = (L_1 \cap L_{\infty})^0 = \{0\}$?

It is interesting to observe that the above phenomenon has its counterpart in the general theory of Banach lattices. Let $E$ 
be a separable Banach lattice satisfying the Radon-Nikodym property (RNP). Then the set $F$ of all $x^* \in E^*$, such that 
the interval $[0, |x^*|]$ is weakly compact is a Banach lattice. Lotz showed (in unpublished preprint \cite{Lo75}) that if $F$ is 
big enough, i.e., the topology $\sigma(E, F)$ is Hausdorff, then $E = F^*$. Hence, $F$ is a natural candidate as the predual 
of $E$. In \cite{Ta81} Talagrand, however, motivated by above results, has constructed a separable Banach lattice being 
a dual space (and hence with RNP) such that for each $x^* \in E^*, x^* \geq 0$, the interval $[0, x^*]$ is not weakly compact.

To see that the space $Ces_{\infty}$ may be regarded as a natural ``function" counterpart of Talagrand's example (which seems to 
be rather artificial) we present the following simple assertion.

\begin{proposition}\label{Cor12}
Let $X$ be an ideal Banach space on $[0, 1], x_0 \in X, x_0 \geq 0$. Then the interval $[0, x_0]$ is weakly compact in $X$ if and 
only if $x_0 \in X^0$.
\end{proposition}
\proof Firstly, let $[0, x_0]$ be weakly compact in $X$. On the contrary, assume that $x_0 \notin X^0$. Then there are a sequence of 
sets $\{A_n\}_{n=1}^{\infty}, A_1 \supset A_2 \supset A_3 \supset \ldots, \bigcap_{n=1}^{\infty} A_n = \emptyset$, and $\varepsilon > 0$ 
such that 
\begin{equation} \label{51}
\| x_0 \chi_{A_n} \|_X \geq \varepsilon.
\end{equation}
Since $x_0 \chi_{A_n} \in [0, x_0]$, by hypothesis, we can find a subsequence $\{x_0 \chi_{A_{n_k}} \}_{k=1}^{\infty}$ such that 
$x_0 \chi_{A_{n_k}}  \rightarrow y$ weakly in $X$. Then, by \cite[Lemma 10.4.1]{KA77}, we get that $\| x_0 \chi_{A_{n_k}} \|_X \rightarrow 0$ 
as $k \rightarrow \infty$. This contradicts (\ref{51}).

Conversely, let $x_0 \in X^0$. Clearly, we have $[0, x_0] \subset X^0$. Therefore, by \cite[Lemma 10.4.2]{KA77}, the interval $[0, x_0]$ 
is weakly compact in $X^0$, i.e., with respect to the topology generated in $X^0$ by the space $(X^0)^* = X^{\prime}$. Since 
$ X^* = X^{\prime} \bigoplus X_s^{\prime}$, where $X_s^{\prime}$ consists of all singular functionals $f$ such that $f{\big |}_{X^0} = 0$ (see 
\cite[Theorem 10.3.6]{KA77}), we get that $[0, x_0]$ is weakly compact in $X$ as well.
\endproof

\begin{remark}\label{1new}
In particular, from Proposition \ref{Cor12}, it follows that the above Lotz's result cannot be applied to $Ces_{\infty}$. 
In fact,  $(Ces_{\infty})^*= \widetilde{L^1}\oplus S$, where $S$ is the space of singular functionals, and, since singular 
functionals are not comparable with regular ones, each interval $[0,|x^*|]\subset Ces_{\infty}^*$ is either non-weakly compact 
or is of the form $[0,|s|]$ with $s\in S$. Therefore, the set $F$ of all $x^* \in (Ces_{\infty})^*$ with the weakly compact interval 
$[0, |x^*|]$ is contained in $S$ and the topology $\sigma(Ces_{\infty}, F)$ fails to be Hausdorff, because singular 
functionals vanish on absolutely continuous  elements.
\end{remark}

\begin{remark}\label{2new}
From results obtained in this section it follows that $(\bigoplus_{k=1}^{\infty} L_1[0,1])_{l_\infty}$ is isomorphic to a dual space. 
On the other hand, the unit sphere of this space does not contain extreme points and, hence, it is not isometric to a dual space. 
Thanks to the well-known Davis-Johnson result \cite{DJ737} we know that each nonreflexive Banach space can be renormed 
so that to be nonisometric to a dual one. At the same time, the proof presented in \cite{DJ737} does not concern any information on 
extreme points of the unit sphere of the space derived by suitable renorming. Therefore, having in mind the above example of 
$(\bigoplus_{k=1}^{\infty} L_1[0,1])_{l_\infty}$, we can ask if each nonreflexive Banach space may be renormed so that its unit 
sphere will not contain any extreme points? 
\end{remark}

Since $Ces_{\infty}(I) \simeq X^*$, where $X$ is a Banach space, and it contains a complemented subspace isomorphic to $L_1[0, 1]$, 
then according to the above-mentioned Hagler-Stegall result $Ces_{\infty}(I)$ contains a complemented subspace isomorphic to 
$C[0, 1]^*$, i.e., to the space ${\mathcal M}[0, 1]$ of all regular Borel measures on $[0, 1]$ of finite variation. 
We would like to conclude the paper by presenting the following stronger result, which was noticed by Micha\l \  Wojciechowski and which 
is included here with his kind permission. 

\begin{theorem}\label{measures} \label{Thm12}
The spaces $Ces_{\infty}(I)$, where $I = [0, 1]$ or $[0, \infty)$, is isomorphic to the space $(\bigoplus_{k=1}^{\infty} {\mathcal M}[0,1])_{l_\infty}$.  
\end{theorem}
\proof
At first,  by Miljutin's theorem (cf. \cite[p. 94]{AK06}), we know that $C[0,1] \simeq C(\mathbb{T})$, where  
$\mathbb{T}$ is the unit circle. Since also ${\mathcal M}[0,1]\simeq {\mathcal M}(\mathbb{T})$ and
$L^1[0,1]\simeq L^1(\mathbb{T})$, we can regard all spaces on $\mathbb{T}$ instead of $[0,1]$. 

By Theorem \ref{Thm5}, it is sufficient to prove that the spaces $(\bigoplus_{k=1}^{\infty} L_1(\mathbb{T}))_{l_\infty}$ and 
$(\bigoplus_{k=1}^{\infty} {\mathcal M}(\mathbb{T}))_{l_\infty}$ are isomorphic. Since both spaces are isomorphic to their squares, we 
may again apply Pe{\l}czy\'nski decomposition argument. Clearly, $(\bigoplus_{k=1}^{\infty} L_1(\mathbb{T}))_{l_\infty}$ is isomorphic 
to a complemented subspace of $(\bigoplus_{k=1}^{\infty} {\mathcal M}(\mathbb{T}))_{l_\infty}$. So, we need to check only that, conversely, 
$(\bigoplus_{k=1}^{\infty} {\mathcal M}(\mathbb{T}))_{l_\infty}$  is isomorphic to a complemented subspace of 
$(\bigoplus_{k=1}^{\infty} L_1(\mathbb{T}))_{l_\infty}$.

Let $\{K_n\}_{n=1}^{\infty}$ be the Fejer kernel and let $\{N_i\}_{i=1}^{\infty}$ be a sequence of pairwise disjoint infinite subsets 
of positive integers such that $\sum_{i=1}^{\infty} N_i = \mathbb N$. For every $i = 1, 2, \ldots$ define the operator
$K^i: {\mathcal M}(\mathbb{T}) \rightarrow (\bigoplus_{k=1}^{\infty} L_1(\mathbb{T}))_{l_\infty}$ as follows:
$$
K^i({\mu}): = (K_n*\mu)_{n \in N_i}\;\;\mbox{for every}\;\;\mu \in {\mathcal M}(\mathbb{T}).$$
Then, $\| K^i\| = 1$ and if $N_i = \{ n_j^i\}_{j = 1}^{\infty}$, 
$n_1^i < n_2^i < \ldots$, then $K_{n_j^i}*\mu \rightarrow \mu$ as $j \rightarrow \infty$ weakly* in ${\mathcal M}(\mathbb{T})$ for each 
$i = 1, 2, \ldots$. Hence, $K: = \bigoplus_{i=1}^{\infty} K^i$ is an injective operator from 
$(\bigoplus_{i=1}^{\infty} {\mathcal M}(\mathbb{T}))_{l_\infty}$ into the space
$$
(\bigoplus_{k=1}^{\infty} L_1(\mathbb{T}))_{l_\infty} \, \simeq \, \Big(\bigoplus_{i=1}^{\infty} (\bigoplus_{n \in N_i} L_1(\mathbb{T}))_{l_\infty} \Big)_{l_{\infty}}.
$$
Denoting by $Y$ the image of $K$, we prove that it is complemented in the latter space.

Let $\mathcal U$ be a free ultrafilter. For a given $\{f_k\} \in (\bigoplus_{k=1}^{\infty} L_1(\mathbb{T}))_{l_\infty}$ and any $i = 1, 2, \ldots$ define
the functional $g_i^* \in C(\mathbb{T})^*$ by
$$
\langle g_i^*, g \rangle: = \lim_{\mathcal U} \langle f_{n_j^i}, g \rangle, ~ g \in C(\mathbb{T}).
$$
Since $\| \{f_k\}\| = \sup_{k \in \mathbb N} \| f_k\|_{L_1} < \infty$, then $g_i^*$ is a well-defined, linear and bounded functional.
Therefore, by the Riesz representation theorem, for every $i = 1, 2, \ldots$ there is a measure $\mu_i \in {\mathcal M}(\mathbb{T})$ such that
$\langle g_i^*, g \rangle = \langle \mu_i, g \rangle$ for each $g \in C(\mathbb{T})$. Setting $P(\{f_k\}): = \{\mu_i\}$, we see that $P$ is a linear 
bounded operator from $(\bigoplus_{k=1}^{\infty} L_1(\mathbb{T}))_{l_\infty}$ into $(\bigoplus_{k=1}^{\infty} {\mathcal M}(\mathbb{T}))_{l_\infty}$. 
It remains only to show that the composition $KP$ is a projection from $(\bigoplus_{k=1}^{\infty} L_1(\mathbb{T}))_{l_\infty}$ onto $Y$. 
In fact, suppose that $\{f_k\} \subset Y$. Then $f_{n_j^i} = K_{n_j^i} * \mu_i, i, j = 1, 2, \ldots$, where 
$\{\mu_i\} \in (\bigoplus_{i=1}^{\infty} {\mathcal M}(\mathbb{T}))_{l_\infty}$, and we have
$$
\lim_{\mathcal U} \langle f_{n_j^i}, g \rangle = \lim_{\mathcal U} \langle K_{n_j^i} * \mu_i, g \rangle = 
\lim_{\mathcal U} \langle \mu_i, K_{n_j^i} * g \rangle =\lim_{j\to\infty}\langle \mu_i, K_{n_j^i} * g \rangle=\langle \mu_i, g \rangle
$$
for every $g \in C(\mathbb{T})$. Thus, $KP\{f_k\} = \{f_k\}$ if $\{f_k\} \in Y$, and the proof is complete. 
\endproof

{\bf Acknowledgements} The second named author is very grateful to Professors Yves Raynaud and Micha\l\  Wojciechowski 
for valuable suggestions, advices and remarks concerning results from the last section.


\end{document}